\documentclass[a4paper, oneside, english,reqno]{amsart}

\usepackage[utf8]{inputenx}
\usepackage{amsthm, amsmath, amssymb, amsfonts, mathrsfs, mathtools, stmaryrd}
\usepackage{babel, textcomp, url, enumerate, latexsym, graphicx, varioref, hyperref, multirow, layout, siunitx, booktabs}
\usepackage{pdflscape}
\usepackage{verbatim}
\usepackage{geometry}

\usepackage{tikz, babel, rotating, tikz-cd, pigpen}
\usetikzlibrary{matrix, arrows, decorations.pathmorphing}
\usepackage{cleveref}

\numberwithin{equation}{section}

\usepackage[all]{xy}
\usepackage{microtype}

\DeclarePairedDelimiter{\p}{\lparen}{\rparen}

\DeclarePairedDelimiter{\ip}{\langle}{\rangle}

\theoremstyle{plain}
\newtheorem{theorem}{Theorem}[section]
\newtheorem{lemma}[theorem]{Lemma}

\newtheorem{prop}[theorem]{Proposition}

\newtheorem{corollary}[theorem]{Corollary}

\theoremstyle{definition}
\newtheorem{definition}[theorem]{Definition}

\theoremstyle{remark}
\newtheorem{remark}[theorem]{Remark}

\newcommand{\wt}{\widetilde}

\newcommand{\ol}{\overline}

\newcommand{\defeq}{\vcentcolon=}

\newcommand{\Nis}{\mathrm{Nis}}
\newcommand{\Zar}{\mathrm{Zar}}
\newcommand{\op}{\mathrm{op}}

\newcommand{\rmc}{\mathrm{c}}
\newcommand{\sspt}{\mathbf{1}}
\newcommand{\A}{\mathbf{A}}

\newcommand{\G}{\mathbf{G}}

\newcommand{\K}{\mathbf{K}}

\newcommand{\PP}{\mathbf{P}}

\newcommand{\Z}{\mathbf{Z}}

\newcommand{\frakm}{\mathfrak{m}}

\newcommand{\sfE}{\mathsf{E}}

\newcommand{\scrD}{\mathscr{D}}

\newcommand{\scrF}{\mathscr{F}}

\newcommand{\scrH}{\mathscr{H}}

\newcommand{\scrL}{\mathscr{L}}

\newcommand{\scrU}{\mathscr{U}}
\newcommand{\scrV}{\mathscr{V}}

\newcommand{\scrX}{\mathscr{X}}

\newcommand{\scrZ}{\mathscr{Z}}

\newcommand{\calA}{\mathcal{A}}

\newcommand{\calO}{\mathcal{O}}
\newcommand{\calP}{\mathcal{P}}

\newcommand{\calV}{\mathcal{V}}

\newcommand{\DM}{\mathbf{DM}}

\newcommand{\Ab}{\mathrm{Ab}}

\newcommand{\Sch}{\mathrm{Sch}}

\newcommand{\Sm}{\mathrm{Sm}}
\newcommand{\SmOp}{\mathrm{SmOp}}

\newcommand{\SH}{\mathbf{SH}}

\newcommand{\PSh}{\mathrm{PSh}}

\newcommand{\red}{\mathrm{red}}
\newcommand{\can}{\mathrm{can}}
\newcommand{\pair}{\mathrm{pr}}
\newcommand{\rmMW}{\mathrm{MW}}
\newcommand{\rmM}{\mathrm{M}}

\DeclareMathOperator{\im}{im}

\DeclareMathOperator{\Hom}{Hom}

\DeclareMathOperator{\coker}{coker}
\DeclareMathOperator{\Spec}{Spec}
\DeclareMathOperator{\pr}{pr}

\DeclareMathOperator{\supp}{supp}
\DeclareMathOperator{\id}{id}

\DeclareMathOperator{\ord}{ord}

\DeclareMathOperator{\Div}{div}

\DeclareMathOperator{\Cor}{Cor}

\DeclareMathOperator{\Fr}{Fr}
\DeclareMathOperator{\CH}{CH}

\DeclareMathOperator{\rmK}{K}
\DeclareMathOperator{\rmH}{H}
\DeclareMathOperator{\hwtCor}{h{\widetilde{Cor}}}

\newcommand*\cocolon{
        \nobreak
        \mskip6mu plus1mu
        \mathpunct{}
        \nonscript
        \mkern-\thinmuskip
        {:}
        \mskip2mu
        \relax
}

\newcommand{\SheafHom}{\mathscr{H}\kern-3pt om}
\SelectTips{cm}{10}

\address{H{\aa}kon Kolderup, University of Oslo, Postboks 1053, Blindern, 0316 Oslo, Norway}
\email{\href{mailto:haakoak@math.uio.no}{haakoak@math.uio.no}}

\begin{document}

\title[Homotopy invariance of $\rmMW$-sheaves]{Homotopy invariance of Nisnevich sheaves with \\Milnor--Witt transfers}
\author{Håkon Kolderup}

\subjclass[2010]{14F05, 14F20, 14F35, 14F42, 19E15}
\keywords{Motives, Milnor--Witt $\rmK$-theory, Chow--Witt groups, motivic homotopy theory}

\maketitle

\begin{abstract}
    \noindent
The category of finite Milnor--Witt correspondences, introduced by Calmès and Fasel, provides a new type of correspondences closer to the motivic homotopy theoretic framework than Suslin--Voevodsky's correspondences. A fundamental result of the theory of ordinary correspondences concerns homotopy invariance of sheaves with transfers, and in the present paper we address this question in the setting of Milnor--Witt correspondences. Employing techniques due to Druzhinin, Fasel--Østvær and Garkusha--Panin, we show that homotopy invariance of presheaves with Milnor--Witt transfers is preserved under Nisnevich sheafification.
\end{abstract}

\tableofcontents

\section{Introduction}
A stepping stone toward Voevodsky's construction of the derived category of motives $\DM(k)$ \cite{Voe-motives} is the notion of finite correspondences between smooth $k$-schemes. Such correspondences are in a certain sense multivalued functions taking only finitely many values. By considering finite correspondences instead of ordinary morphisms of schemes, one performs a linearization which allows for extra elbowroom and flexibility, and which in turn makes it possible to prove strong theorems. One of the ``fundamental theorems'' in the theory of correspondences concerns homotopy invariance, and is crucial for constructing the theory of motives.

\begin{theorem}[\protect{\cite[Theorem 5.6]{Voe-hty-inv}}]\label{thm:strict-hty-inv}
For any homotopy invariant presheaf $\scrF$ on the category $\Cor_k$ of finite correspondences, the associated Nisnevich sheaf $\scrF_\Nis$ is also homotopy invariant.
\end{theorem}

In \cite{Calmes-Fasel}, Calmès and Fasel introduce a new type of correspondences called finite Milnor--Witt correspondences (or finite $\rmMW$-correspondences for short). Milnor--Witt correspondences provide a setting that is closer to the motivic homotopy theoretic framework than Suslin--Voevodsky's correspondences; for example, the zero-line of sheaves of motivic homotopy groups of the sphere spectrum do not admit ordinary transfers, but they do admit $\rmMW$-transfers \cite{Calmes-Fasel}. Roughly speaking, a finite $\rmMW$-correspondence amounts to an ordinary finite correspondence along with an unramified quadratic form defined on the function field of each irreducible component of the support of the correspondence. We briefly recall some results in the theory of $\rmMW$-correspondences below. Our present goal is to prove a homotopy invariance result similar to \Cref{thm:strict-hty-inv} for sheaves with $\rmMW$-transfers:

\begin{theorem}\label{thm:strict-MW-hty-inv}
Let $k$ be a field of characteristic\footnote{The assumption on the characteristic is there because Milnor--Witt correspondences are currently not defined over nonperfect fields. The only place where this assumption is used is in \Cref{section:A1K} where we need to consider Milnor--Witt correspondences defined over function fields of smooth $k$-schemes, which may in general be nonperfect. Otherwise, all excision results are valid for infinite perfect fields of characteristic different from $2$.} $0$. Then, for any homotopy invariant presheaf $\scrF$ on the category $\wt\Cor_k$ of finite $\rmMW$-correspondences, the associated Nisnevich sheaf $\scrF_\Nis$ is also homotopy invariant.
\end{theorem}

We note that this result is already known by work of D\'eglise and Fasel \cite[Theorem 3.2.9]{MW-cplx}. Their proof uses the fact that there is a functor $\Fr_*(k)\to \wt\Cor_k$ from the category of framed correspondences to $\rmMW$-correspondences. As the analog of \Cref{thm:strict-MW-hty-inv} is known for framed correspondences by work of Garkusha and Panin \cite{hty-inv}, it follows that the desired result also holds for $\wt\Cor_k$. The purpose of this paper is to give a more direct proof by using geometric input provided in \cite[§13]{hty-inv} to produce homotopies in $\wt\Cor_k$. Along the way we obtain results on $\rmMW$-correspondences of independent interest. The proof strategy is due to Druzhinin \cite{Druzhinin} and Garkusha--Panin \cite{hty-inv}, and uses techniques developed in \cite{MW-cancel}.

\subsection*{Recollections on Milnor--Witt correspondences}
The Milnor--Witt $\rmK$-groups $\rmK_n^{\rmMW}(k)$ of a field $k$ arose in the context of motivic stable homotopy groups of spheres. More precisely, in \cite[Theorem 6.4.1]{Morel-sphere-spt} Morel established isomorphisms
\begin{align}
\pi_{n,n}\sspt\cong \rmK^{\rmMW}_{-n}(k)\label{eq:MW-iso}
\end{align}
for all $n\in\Z$, where $\sspt\in\SH(k)$ denotes the sphere spectrum. The groups $\rmK_n^{\rmMW}(k)$ admit a description in terms of generators and relations:

\begin{definition}[Hopkins--Morel]
Let $k$ be a field. The \emph{Milnor--Witt $\rmK$-theory} $\rmK_*^{\rmMW}(k)$ of the field $k$ is the graded associative $\Z$-algebra with one generator $[a]$ for each unit $a\in k^\times$, of degree $+1$, and one generator $\eta$ of degree $-1$, subject to the following relations:
\begin{enumerate}
\item[(1)] $[a][1-a]=0$ for any $a\in k^\times\setminus\{1\}$\quad(Steinberg relation).
\item[(2)] $\eta[a]=[a]\eta$\quad ($\eta$-commutativity).
\item[(3)] $[ab]=[a]+[b]+\eta[a][b]$\quad (twisted $\eta$-logarithmic relation).
\item[(4)] $(2+\eta[-1])\eta=0$\quad (hyperbolic relation).
\end{enumerate}
We let $\rmK_n^{\rmMW}(k)$ denote the $n$-th graded piece of $\rmK_*^{\rmMW}(k)$. The product $[a_1]\cdots[a_n]\in \rmK_n^{\rmMW}(k)$ may also be denoted by $[a_1,\dots,a_n]$.
\end{definition}

Under the isomorphism (\ref{eq:MW-iso}) above, the element $[a]\in \rmK_1^{\rmMW}(k)$ corresponds to a class $[a]\in\pi_{-1,-1}\sspt$. A representative for $[a]$ is given by the pointed map 
\[[a]\colon \Spec(k)_+\to (\G_m,1)\] 
sending the non-basepoint to the point $a\in\G_m$. On the other hand, the element $\eta\in \rmK_{-1}^{\rmMW}(k)$ corresponds to the motivic Hopf map $\eta\in\pi_{1,1}\sspt$ represented by the natural projection \cite[§6]{Morel-sphere-spt}
\[\eta\colon \A^2\setminus0\to \PP^1.\]
As the sphere spectrum is initial in the category of motivic ring spectra, the homotopy groups $\pi_{p,q}\sfE$ of a motivic ring spectrum $\sfE$ inherits the relations of $\pi_{p,q}\sspt$ via the unit map $\sspt\to\sfE$. Thus Milnor--Witt $\rmK$-theory is a fundamental object in motivic homotopy theory. In \cite{Calmes-Fasel}, Calmès and Fasel employ sheaves of Milnor--Witt $\rmK$-theory to set up the theory of $\rmMW$-correspondences. Based on the fact that the group $\Cor_k(X,Y)$ of finite correspondences from $X$ to $Y$ can be expressed as a colimit of Chow groups with support,
\begin{align*}
\Cor_k(X,Y)&=\varinjlim_{T\in\calA(X,Y)}\rmH_T^{d_Y}(X\times Y,\K_{d_Y}^\rmM)\\
&=\varinjlim_{T\in\calA(X,Y)}\CH_T^{d_Y}(X\times Y),
\end{align*}
Calmès and Fasel replace Milnor $\rmK$-theory (and Chow groups) with (twisted) Milnor--Witt $\rmK$-theory (and Chow--Witt groups), and define the group of finite $\rmMW$-correspondences from $X$ to $Y$ as
\begin{align*}
\wt\Cor_k(X,Y)&\defeq \varinjlim_{T\in\calA(X,Y)}\rmH_T^{d_Y}(X\times Y,\K_{d_Y}^{\rmMW},p_Y^*\omega_{Y/k})\\
&=\varinjlim_{T\in\calA(X,Y)}\wt\CH_T^{d_Y}(X\times Y,p_Y^*\omega_{Y/k}),
\end{align*}
where $p_Y\colon X\times Y\to Y$ is the projection. Here $Y$ is assumed to be equidimensional of dimension $d_Y$, and $\calA(X,Y)$ is the partially ordered set of closed subsets $T$ of $X\times Y$ such that each irreducible component of $T$ (with its reduced structure) is finite and surjective over $X$. Moreover, $\K_n^{\rmMW}$ is the $n$-th unramified Milnor--Witt $\rmK$-theory sheaf, as defined in \cite[§5]{Morel-over-a-field}. We note that the Nisnevich cohomology groups $\rmH^p(X,\K_q^{\rmMW},\scrL)$ of the Milnor--Witt sheaf $\K_q^{\rmMW}(\scrL)$ twisted by a line bundle $\scrL$ can be computed using the Rost--Schmid complex \cite[Chapter 5]{Morel-over-a-field}, which provides a flabby resolution of $\K_q^{\rmMW}(\scrL)$. Recall that the $p$-th term of the Rost--Schmid complex is given by
\[
C^p(X,\K_q^{\rmMW},\scrL)\defeq\bigoplus_{x\in X^{(p)}}\rmK_{q-p}^{\rmMW}(k(x),\wedge^p(\frakm_x/\frakm_x^2)^\vee\otimes_{k(x)}\scrL_x),
\]
where $X^{(p)}$ denotes the set of codimension $p$-points of $X$. We let $\wt\Cor_k$ denote the category of finite $\rmMW$-correspondences. The category $\wt\Cor_k$ is symmetric monoidal, and comes equipped with an embedding $\Sm_k\to\wt\Cor_k$ from smooth $k$-schemes, as well as a forgetful functor $\wt\Cor_k\to\Cor_k$ to Suslin--Voevodsky's correspondences; see \cite{Calmes-Fasel} for details.

Let $\wt\PSh(k)$ denote the category of presheaves with $\rmMW$-transfers, i.e., additive presheaves of abelian groups $\scrF\colon\wt\Cor_k^\op\to\Ab$. As noted in \cite{Calmes-Fasel}, there are more presheaves on $\wt\Cor_k$ than on $\Cor_k$. One example is of course provided by the sheaves $\K_*^{\rmMW}$, which admit $\rmMW$-transfers but not ordinary transfers \cite{Calmes-Fasel}. Among the various presheaves with $\rmMW$-transfers, the homotopy invariant ones will be of most interest to us.

\begin{definition}
A presheaf $\scrF\in\wt\PSh(k)$ with $\rmMW$-transfers is \emph{homotopy invariant} if for each $X\in\Sm_k$, the projection $p\colon X\times\A^1\to X$ induces an isomorphism $p^*\colon\scrF(X)\xrightarrow{\cong}\scrF(X\times\A^1)$. Equivalently, the zero section $i_0\colon X\to X\times\A^1$ induces an isomorphism $i^*_0\colon\scrF(X\times\A^1)\xrightarrow{\cong}\scrF(X)$.
\end{definition}

Let us also mention that by \cite[Lemma 1.2.10]{MW-cplx}, the Nisnevich sheaf $\scrF_\Nis$ associated to a presheaf $\scrF\in\wt\PSh(k)$ comes equipped with a unique $\rmMW$-transfer structure. This result follows essentially from \cite[Lemma 1.2.6]{MW-cplx}, which states that if $p\colon U\to X$ is a Nisnevich covering of a smooth $k$-scheme $X$, and if $\wt\rmc(X)$ denotes the representable presheaf $\wt\rmc(X)(Y)\defeq\wt\Cor_k(Y,X)$, then the \v Cech-complex $\wt\rmc(U_X^\bullet)\to\wt\rmc(X)\to0$ is exact on the associated Nisnevich sheaves.

\subsection*{Extending presheaves to essentially smooth schemes}
In this paper we will consider two closely related ways to extend presheaves on $\wt\Cor_k$ to essentially smooth schemes over $k$. This allows us to formulate statements also about local schemes or henselian local schemes.

\begin{enumerate}
\item The first method is the standard way of defining the value of a presheaf on limits of schemes as a colimit of the presheaf values, and will be used in Sections \ref{section:injectivity}--\ref{section:hty-inv}. We briefly recall some details on this matter, following \cite[§5.1]{Calmes-Fasel}. Let $\calP$ be the category consisting of projective systems $((X_\lambda)_{\lambda\in I},f_{\lambda\mu})$ such that $X_\lambda\in\Sm_k$ and such that the transition morphisms $f_{\lambda\mu}\colon X_\lambda\to X_\mu$ are affine and étale. By \cite[§5.1]{Calmes-Fasel}, the limit of such a projective system exists in the category of schemes $\Sch$. Moreover, the functor $\calP\to\Sch$ sending a projective system to its limit defines an equivalence of categories between $\calP$ and the full subcategory $\ol\Sm_k$ of $\Sch$ consisting of schemes over $k$ that are limits of projective systems from $\calP$ \cite[§5.1]{Calmes-Fasel}. 

Now let $\scrF$ be a presheaf on $\Sm_k$. We can extend $\scrF$ to a presheaf $\ol\scrF$ on $\ol\Sm_k$ by setting $\ol\scrF((X_\lambda)_{\lambda\in I})\defeq\varinjlim_{\lambda\in I}\scrF(X_\lambda)$. By \cite[§5.1]{Calmes-Fasel} this gives a well defined presheaf on $\ol\Sm_k$ that coincides with $\scrF$ when restricted to $\Sm_k$. In particular, we can extend the presheaf $\wt\Cor_k(-,X)$ to $\ol\Sm_k$.

The above construction can furthermore be carried out for Chow--Witt groups with support. Roughly speaking, we can define a category $\wt\calP$ consisting of projective systems of triples $(X_\lambda,Z_\lambda,\scrL_\lambda)$ of a smooth $k$-scheme $X_\lambda$, a closed subscheme $Z_\lambda$ of $X_\lambda$ and a line bundle $\scrL_\lambda$ on $X_\lambda$. If the limit $(X,Z,\scrL)$ of such a projective system is such that $X$ is regular, then the pullback induces an isomorphism \cite[Lemma 5.7]{Calmes-Fasel}
\[
\varinjlim_{\lambda}\wt\CH_{Z_\lambda}^*(X_\lambda,\scrL_\lambda)\xrightarrow{\cong}\wt\CH_Z^*(X,\scrL).
\]
This allows us to pass to Chow--Witt groups of local schemes $U$ in order to produce MW-correspondences on $U$, which will be needed in Sections \ref{section:injectivity}--\ref{section:Nis-surj}. However, in order to unburden our notation we may drop the bar both from $\ol\scrF$ and $\ol\Sm_k$ when evaluating presheaves on limits of schemes.

\item A second method of extending presheaves will be carried out in \Cref{section:A1K} in order to show that certain results that hold for open subsets of $\A^1_k$ are also valid for open subsets of $\A^1_K$, where $K$ is some finitely generated field extension $K$ of the ground field $k$. This trick was suggested to the author by I. Panin, and involves extending a presheaf on $\wt\Cor_k$ to a certain full subcategory of $\wt\Cor_K$. See \Cref{section:A1K} for details.
\end{enumerate}

\subsection*{Outline} 
In \Cref{section:pairs} we establish some notation and collect a few lemmas needed later on. 

In \Cref{section:cartier} we review how Cartier divisors give rise to finite $\rmMW$-correspondences, following \cite{MW-cancel}. This gives a procedure to construct desired homotopies in the later sections. 

In \Cref{section:Zar-excision} we prove the first main ingredient of the proof of \Cref{thm:strict-MW-hty-inv}, which is a Zariski excision result for $\rmMW$-presheaves. More precisely, in \Cref{thm:Zar-excision} we show that if $V\subseteq U\subseteq\A^1$ are two Zariski open neighborhoods of a closed point $x\in\A^1$, then the inclusion $i\colon V\hookrightarrow U$ induces an isomorphism\footnote{We show in \Cref{section:inj-aff-line} that the restriction maps $\scrF(U)\to\scrF(U\setminus x)$ and $\scrF(V)\to\scrF(V\setminus x)$ are injective, justifying the notation used in the formulation of Zariski excision.}
\[i^*\colon \dfrac{\scrF(U\setminus x)}{\scrF(U)}\xrightarrow{\cong} \dfrac{\scrF(V\setminus x)}{\scrF(V)}\]
for any homotopy invariant $\scrF\in\wt\PSh(k)$. The proof of Zariski excision consists of producing left and right inverses in $\wt\Cor_k$ of $i$ up to homotopy. This is done in Sections \ref{section:Zar-inj} and \ref{section:Zar-surj}.

In \Cref{section:A1K}, we extend the results of \Cref{section:Zar-excision} to open subsets of $\A^1_K$, where $K$ is a finitely generated field extension of the ground field $k$. 

In \Cref{section:injectivity} we prove a ``moving lemma'' for $\rmMW$-correspondences (see \Cref{thm:inj-loc-sch}), which can be informally stated as follows. Let $X\in\Sm_k$, and pick a closed point $x\in X$ along with a closed subscheme $Z\subsetneq X$ containing the point $x$. Then, up to $\A^1$-homotopy, we are able to ``move the point $x$ away from $Z$'' using $\rmMW$-correspondences. See \Cref{section:injectivity} for more details.  

In \Cref{section:Nis-excision} we prove the last main ingredient of the proof of \Cref{thm:strict-MW-hty-inv}, namely a Nisnevich excision result. The situation is as follows. Given an elementary distinguished Nisnevich square
\[\begin{tikzcd}
V'\ar{r}\ar{d} & X'\ar[d,"\Pi"]\\
V\ar{r} & X
\end{tikzcd}\]
with $X$ and $X'$ affine and $k$-smooth, let $S\defeq (X\setminus V)_\red$ and $S'\defeq(X'\setminus V')_\red$. Assume in addition that $S$ is $k$-smooth. Suppose that $x\in S$ and $x'\in S'$ are two points satisfying $\Pi(x')=x$, and put $U\defeq\Spec(\calO_{X,x})$ and $U'\defeq\Spec(\calO_{X',x'})$. Then the map $\Pi$ induces an isomorphism\footnote{It follows from \Cref{thm:inj-loc-sch} that the restriction maps $\scrF(U)\to\scrF(U\setminus S)$ and $\scrF(U')\to\scrF(U'\setminus S')$ are injective, justifying the notation used in the formulation of Nisnevich excision. See \Cref{section:Nis-excision} for details.}
\[
\Pi^*\colon \dfrac{\scrF(U\setminus S)}{\scrF(U)}\xrightarrow{\cong} \dfrac{\scrF(U'\setminus S')}{\scrF(U')}
\]
for any homotopy invariant $\scrF\in\wt\PSh(k)$. Again the proof consists of producing left and right inverses to $\Pi$ up to homotopy, which is done in Sections \ref{section:Nis-inj} and \ref{section:Nis-surj}. 

Finally, in \Cref{section:hty-inv} we will see how homotopy invariance of the associated Nisnevich sheaf $\scrF_\Nis$ follows from the above results.

\subsection*{Conventions}Throughout we will assume that $k$ is an infinite perfect field of characteristic different from $2$. In Sections \ref{section:A1K} and \ref{section:hty-inv}, $k$ is furthermore assumed to be of characteristic $0$. We let $\Sm_k$ denote the category of smooth separated schemes of finite type over $k$. All undecorated fiber products mean fiber product over $k$. Throughout, the symbols $i_0$ and $i_1$ will denote the rational points $i_0,i_1\colon\Spec(k)\to \A^1$ given by $0$ and $1$, respectively. 

We will frequently abuse notation and write simply $f\in\wt\Cor_k(X,Y)$ for $\wt\gamma_f$, where $\wt\gamma_f$ is the image of a morphism of schemes $f\colon X\to Y$ under the embedding $\wt\gamma\colon \Sm_k\to\wt\Cor_k$ of \cite[§4.3]{Calmes-Fasel}. We let $\sim_{\A^1}$ denote $\A^1$-homotopy equivalence. Following Calmès--Fasel \cite{Calmes-Fasel}, if $p_Y\colon X\times Y\to Y$ is the projection, we may write $\omega_Y$ as shorthand for $p_Y^*\omega_{Y/k}$ if no confusion is likely to arise. Note that $\omega_Y$ is then canonically isomorphic to $\omega_{X\times Y/X}$. In general, given a morphism of schemes $f\colon X\to Y$ we write $\omega_f\defeq\omega_{X/k}\otimes f^*\omega_{Y/k}^\vee$.

\subsection*{Acknowledgments} I am very grateful to Ivan Panin for sharing with me his knowledge on homotopy invariance, and for suggesting to me the trick of \Cref{section:A1K}. Furthermore, I am grateful to Jean Fasel for valuable comments and suggestions, and to Paul Arne Østvær for proofreading and for suggesting the problem to me. I thank Institut Mittag-Leffler for the kind hospitality during spring 2017. Finally, I thank the anonymous referee for many helpful comments and remarks.

The work on this paper was supported by the RCN Frontier Research Group Project no. 250399.

\section{Pairs of Milnor--Witt correspondences}\label{section:pairs}
We will frequently encounter the situation of a pair $U\subseteq X$ of schemes, and we will be led to study the associated quotient $\scrF(U)/\im(\scrF(X)\to\scrF(U))$ for a given presheaf $\scrF$ on $\wt\Cor_k$. It is therefore notationally  convenient to introduce a category $\wt\Cor_k^\pair$ of pairs of $\rmMW$-correspondences.

Following \cite{hty-inv} we let $\SmOp_k$ denote the category whose objects are pairs $(X,U)$ with $X\in\Sm_k$ and $U$ a Zariski open subscheme of $X$, and whose morphisms are maps $f\colon(X,U)\to(Y,V)$, where $f\colon X\to Y$ is a morphism of schemes such that $f(U)\subseteq V$. Below we extend this notion of morphisms of pairs to $\rmMW$-correspondences.

\begin{definition}[\protect{\cite[Definition 2.3]{hty-inv}}]
Let $\wt\Cor_k^\pair$ denote the category whose objects are those of $\SmOp_k$ and whose morphisms are defined as follows. For $(X,U),(Y,V)\in\SmOp_k$, with open immersions $j_X\colon U\to X$ and $j_Y\colon V\to Y$, let
\[
\wt\Cor_k^\pair((X,U),(Y,V))\defeq\ker\p*{\wt\Cor_k(X,Y)\oplus\wt\Cor_k(U,V)\xrightarrow{j_X^*-(j_Y)_*}\wt\Cor_k(U,Y)}.
\]
Thus a morphism in $\wt\Cor_k^\pair$ is a pair $(\alpha,\beta)$, where $\alpha\in\wt\Cor_k(X,Y)$ and $\beta\in\wt\Cor_k(U,V)$, such that the diagram
\[\begin{tikzcd}
X\ar[r,"\alpha"] & Y\\
U\ar[u,"j_X"]\ar[r,"\beta"] & V\ar[u,swap,"j_Y"]
\end{tikzcd}\]
commutes in $\wt\Cor_k$. Composition in $\wt\Cor_k^\pair$ is defined by $(\alpha,\beta)\circ(\gamma,\delta)\defeq(\alpha\circ\gamma,\beta\circ\delta)$.
\end{definition}

The category $\SmOp_k$ contains $\Sm_k$ as a full subcategory, the embedding $\Sm_k\to \SmOp_k$ being defined by $X\mapsto (X,\varnothing)$. Moreover, the embedding $\Sm_k\to \SmOp_k$ induces a fully faithful embedding $\wt\Cor_k\to\wt\Cor_k^\pair$ which on morphisms is given by $\alpha\mapsto(\alpha,0)$.

\begin{prop}[\protect{\cite[Construction 2.8]{hty-inv}}]\label{prop:construction}
Suppose that $\scrF$ is a presheaf on $\wt\Cor_k$. For any $(X,U)\in\SmOp_k$, let $\scrF(X,U)\defeq \scrF(U)/\im(\scrF(X)\to\scrF(U))$. Then, for any $(\alpha,\beta)\in\wt\Cor_k^\pair((X,U),(Y,V))$, $\scrF$ induces a morphism
\[
(\alpha,\beta)^*\colon \scrF(Y,V)\to \scrF(X,U).
\]
\end{prop} 

\begin{definition}[\protect{\cite[Definition 2.3]{hty-inv}}]\label{def:hty-cat}
Define the homotopy category $\hwtCor_k$ of $\wt\Cor_k$ as follows. The objects of $\hwtCor_k$ are the same as those of $\wt\Cor_k$, and the morphisms are given by
\begin{align*}
\hwtCor_k(X,Y)&\defeq\wt\Cor_k(X,Y)/\sim_{\A^1}\\
&=\coker\p*{\wt\Cor_k(\A^1\times X,Y)\xrightarrow{i_0^*-i_1^*}\wt\Cor_k(X,Y)}.
\end{align*}
Similarly, let $\hwtCor_k^\pair$ denote the category whose objects are those of $\wt\Cor_k^\pair$, and whose morphisms are given by
\begin{align*}
&\hwtCor_k^\pair((X,U),(Y,V))\defeq\\
&\coker\p*{\wt\Cor_k^\pair(\A^1\times(X,U),(Y,V))\xrightarrow{i_0^*-i_1^*}\wt\Cor_k^\pair((X,U),(Y,V))}.
\end{align*}
Here $\A^1\times(X,U)$ is shorthand for $(\A^1\times X,\A^1\times U)$. If $\alpha\in\wt\Cor_k(X,Y)$ is a finite $\rmMW$-correspondence, we write $\ol\alpha$ for the image of $\alpha$ in $\hwtCor_k(X,Y)$. Similarly, if $(\alpha,\beta)$ is a morphism  in $\wt\Cor_k^\pair$ from $(X,U)$ to $(Y,V)$, write $\ol{(\alpha,\beta)}$ for the image of $\alpha$ in $\hwtCor_k^\pair((X,U),(Y,V))$. Note that a presheaf on $\wt\Cor_k$ is homotopy invariant if and only if it factors through $\hwtCor_k$. Moreover, the embedding $\wt\Cor_k\to\wt\Cor_k^\pair$ induces a fully faithful embedding $\hwtCor_k\to\hwtCor_k^\pair$.
\end{definition}

Next we record a few observations that will come in handy later on:

\begin{lemma}\label{lemma:conn-comp-sum}
Suppose that $\alpha$ is a finite $\rmMW$-correspondence from $X$ to $Y$. Let $T_1,\dots,T_n$ be the connected components of the support $T$ of $\alpha$.
Then, for each $i=1,\dots,n$ there are uniquely determined finite $\rmMW$-correspondences $\alpha_i$ supported on $T_i$ such that $\alpha=\sum_i\alpha_i$.
\end{lemma}

\begin{proof}
Since $\alpha\in\bigoplus_{x\in(X\times Y)^{(d_Y)}}\rmK^{\rmMW}_0(k(x),\wedge^{d_Y}(\frakm_{x}/\frakm_{x}^2)^\vee\otimes(\omega_{Y})_{x})$, we may write $\alpha=\sum_i\alpha_i$ where $\alpha_i$ is supported on $T_i$. To conclude we must show that $\alpha_i\in\wt\CH_{T_i}^{d_Y}(X\times Y,\omega_{Y})$, i.e., that $\partial(\alpha_i)=0$ for all $i$. Now $\partial_x(\alpha_i)=0$ for all $x\in X\times Y$ except perhaps for $x\in T_i$. But since $T_i$ is disjoint from the other $T_j$'s and $\partial(\alpha)=0$ by assumption, we must have $\partial_x(\alpha_i)=0$ also for $x\in T_i$.
\end{proof}

\begin{lemma}\label{lemma:iso}
Let $X$ be a smooth scheme, let $q\in\Z$ be an integer, and let $\scrL$ be a line bundle over $X$. Let $j\colon U\to X$ be a Zariski open subscheme, and suppose that $T\subseteq U$ is a subset which is closed in both $U$ and $X$.
Then the map 
\[
j^*\colon \rmH^p_T(X,\K^{\rmMW}_q,\scrL) \to \rmH^p_T(U,\K^{\rmMW}_q,j^*\scrL)
\] 
is an isomorphism for each $p\in\Z$, with inverse $j_*$, the finite pushforward of \cite[§3]{Calmes-Fasel}.
\end{lemma}

\begin{proof}
The map $j^*$ is an isomorphism by étale excision \cite[Lemma 3.7]{Calmes-Fasel}. The composition $j^*j_*$ is the identity map on the Rost--Schmid complex $C^*_T(U,\K_q^{\rmMW},j^*\scrL)$ supported on $T$, which implies the claim.
\end{proof}

\begin{corollary}\label{cor:supp-open}
Let $X,Y\in\Sm_k$, and let $j\colon V\to Y$ be a Zariski open subscheme. Suppose that $\alpha\in\wt\Cor_k(X,Y)$ is a finite $\rmMW$-correspondence such that $\supp\alpha\subseteq X\times V$. Then there is a unique finite $\rmMW$-correspondence $\beta\in\wt\Cor_k(X,V)$ such that $j\circ\beta=\alpha$. In fact, we have $\beta=(1\times j)^*\alpha$. 
\end{corollary}

\begin{proof}
Let $T\defeq\supp\alpha$, so that by \Cref{lemma:iso} we have mutually inverse isomorphisms
\[
(1\times j)^*\colon\wt\CH^{d_Y}_{T}(X\times Y,\omega_Y)\rightleftarrows \wt\CH_{T}^{d_Y}(X\times V,\omega_V)\cocolon(1\times j)_*
\]
with $\alpha\in\wt\CH^{d_Y}_T(X\times Y,\omega_Y)$. Thus, if $\beta\defeq(1\times j)^*(\alpha)$ then $(1\times j)_*\beta=\alpha$. We conclude the equality $(1\times j)_*\beta=j\circ\beta$ from \cite[Example 4.18]{Calmes-Fasel}.\end{proof}

\begin{lemma}\label{lemma:making-pairs}
Suppose that $j_X\colon U\to X$ and $j_Y\colon V\to Y$ are open subschemes of smooth connected $k$-schemes $X$, $Y$. Assume further that $\alpha\in\wt\Cor_k(X,Y)$ is a finite $\rmMW$-correspondence such that the support $T\defeq \supp\alpha$ satisfies $T\cap(U\times Y)\subseteq U\times V$. Let $\alpha'\defeq(j_X\times j_Y)^*(\alpha)$. Then we have $(\alpha, \alpha')\in\wt\Cor_k^\pair((X,U),(Y,V))$.
\end{lemma}

\begin{proof}
First we show that $\alpha'\in\wt\Cor_k(U,V)$. By contravariant functoriality of Chow--Witt groups we may write $\alpha'=(1\times j_Y)^*(j_X\times 1)^*(\alpha)$. Now $(j_X\times 1)^*(\alpha)=\alpha\circ j_X\in\wt\Cor_k(U,Y)$ by \cite[Example 4.17]{Calmes-Fasel}. By \cite[Lemmas 4.8, 4.10]{Calmes-Fasel}, $\supp(j_X\times 1)^*(\alpha)=T\cap(U\times Y)$ is finite and surjective over $U$. Since $T\cap(U\times Y)\subseteq U\times V$, we have
\[
\alpha'\in\wt\CH_{T\cap(U\times Y)}^{d_Y}(U\times V,(1\times j_Y)^*\omega_{Y}),
\]
where $d_Y\defeq\dim Y$. As $j_Y$ is an open embedding we have $(1\times j_Y)^*\omega_Y\cong\omega_V$; hence $\alpha'$ is a finite $\rmMW$-correspondence from $U$ to $V$.

Next we show that the diagram
\[\begin{tikzcd}
X\ar{r}{\alpha} & Y\\
U\ar{u}{j_X}\ar{r}{\alpha'} & V\ar{u}[swap]{j_Y}
\end{tikzcd}\]
commutes in $\wt\Cor_k$. As $T\cap(U\times Y)=T\cap(U\times V)$, the morphism $(j_X\times 1)^*$ factors as
\[\begin{tikzcd}
\wt\CH_T^{d_Y}(X\times Y,\omega_{Y})\ar{r}{(j_X\times j_Y)^*}\ar{dr}[swap]{(j_X\times1)^*} & \wt\CH_{T\cap(U\times V)}^{d_Y}(U\times V,\omega_{V})\ar{d}{(1\times j_Y)_*}\\
& \wt\CH_{T\cap(U\times Y)}^{d_Y}(U\times Y,\omega_{Y}).
\end{tikzcd}\]
Hence
\[
j_Y\circ\alpha'=(1\times j_Y)_*(j_X\times j_Y)^*(\alpha)=(j_X\times1)^*(\alpha)=\alpha\circ j_X
\]
by \cite[Examples 4.17, 4.18]{Calmes-Fasel}.
\end{proof}

\subsection*{Relative Milnor--Witt correspondences}

For later reference, let us also briefly mention the notion of finite MW-correpondences relative to a base scheme $S\in\Sm_k$.  

\begin{definition}\label{def:rel-wt-Cor}
Let $S\in\Sm_k$ be a smooth $k$-scheme. For any $X,Y\in\Sm_S$, let $p\colon X\times_S Y\to X$ denote the projection, and let $d$ denote the relative dimension of $p$. We define the group of \emph{finite relative $\rmMW$-correspondences} from $X$ to $Y$ as
\[
\wt\Cor_S(X,Y)\defeq\varinjlim_T\wt\CH^d_T(X\times_SY,\omega_p),
\]
where the colimit runs over all closed subsets $T$ of $X\times_SY$ such that each irreducible component of $T_\red$ is finite and surjective over $X$.
\end{definition}

One can show that the groups $\wt\Cor_S(X,Y)$ define the mapping sets of a category $\wt\Cor_S$ of \emph{finite relative $\rmMW$-correspondences}. However, below we will only need the definition of the groups $\wt\Cor_S(X,Y)$, and so we will not pursue the study of the category $\wt\Cor_S$ in further detail here.

\begin{lemma}\label{lemma:rel-MW-to-MW}
Let $S\in\Sm_k$ be a smooth $k$-scheme, and let $X,Y\in\Sm_S$. Then the canonical morphism $f\colon X\times_S Y\to X\times Y$ induces a homomorphism
\[
f_*\colon\wt\Cor_S(X,Y)\to\wt\Cor_k(X,Y)
\]
given as the pushforward on Chow--Witt groups.
\end{lemma}

\begin{proof}
Let $d_Y\defeq\dim Y$ and $d_S\defeq\dim S$. Then the projection $p\colon X\times_SY\to X$ has relative dimension $d_Y-d_S$, and the pushforward map on Chow--Witt groups is given as
\[
f_*\colon\wt\CH_T^{d_Y-d_S}(X\times_SY,\omega_p)\to\wt\CH^{d_Y}_{f(T)}(X\times Y,\omega_Y),
\]
for any admissible subset $T$. Since $f$ is finite, $f(T)$ is also an admissible subset. Hence, composing with the canonical map to the colimit $\wt\Cor_k(X,Y)$ on the right hand side, we obtain the desired homomorphism.
\end{proof}

\section{Milnor--Witt correspondences from Cartier divisors}\label{section:cartier}
Let us recall from \cite[§2]{MW-cancel} how a Cartier divisor gives rise to a finite $\rmMW$-correspondence. Suppose that $X\in\Sm_k$ is a smooth integral $k$-scheme, and let $D=\{(U_i,f_i)\}$ be a Cartier divisor on $X$, with support $|D|$.
We can associate a cohomology class 
\[
\wt\Div(D)\in \rmH^1_{|D|}(X,\K_1^{\rmMW},\calO_X(D))=\wt\CH^1_{|D|}(X,\calO_X(D))
\]
to $D$ as follows. If $x\in X^{(1)}$ is a codimension $1$-point on $X$, choose $i$ such that $x\in U_i$. Consider the element
\[
[f_i]\otimes f_i^{-1}\in \rmK_1^{\rmMW}(k(X),\calO_X(D)\otimes k(X)).
\]

\begin{definition}[\protect{\cite[Definition 2.1.1]{MW-cancel}}]
In the above setting, define
\[
\wt\ord_x(D)\defeq\partial_x([f_i]\otimes f_i^{-1})\in \rmK^{\rmMW}_0(k(x),(\frakm_x/\frakm_x^2)^\vee\otimes_{k(x)}\calO_X(D)_x),
\]
and 
\[
\wt\ord(D)\defeq\sum_{x\in X^{(1)}\cap|D|}\wt\ord_x(D)\in C^1(X,\K_1^{\rmMW},\calO_X(D)).
\]
\end{definition}

By \cite[Lemma 2.1.2]{MW-cancel}, the definition of $\wt\ord_x(D)$ does not depend on the choice of $U_i$, and by \cite[Lemma 2.1.3]{MW-cancel} we have $\partial(\wt\ord(D))=0$. Therefore the element $\wt\ord(D)$ defines a cohomology class in $\wt\CH^1_{|D|}(X,\calO_X(D))$, which we denote by $\wt\Div(D)$.

\begin{lemma}\label{lemma:cartier-sum}
Let $X\in\Sm_k$ be a smooth integral $k$-scheme and suppose that $D$ and $D'$ are two Cartier divisors on $X$ such that
\begin{itemize}
\item the supports of $D$ and $D'$ are disjoint, and 
\item there are trivializations $\chi\colon\calO_X\xrightarrow{\cong}\calO(D)$ and $\chi'\colon\calO_X\xrightarrow{\cong}\calO(D')$.
\end{itemize}
Then $\chi$ and $\chi'$ induce an isomorphism 
\[
\wt\CH^1_{|D+D'|}(X,\calO(D+D'))\cong\wt\CH^1_{|D|}(X,\calO(D))\oplus\wt\CH^1_{|D'|}(X,\calO(D')).
\] 
Under this isomorphism we have the identification $\wt\Div(D+D')=\wt\Div(D)+\wt\Div(D')$.
\end{lemma}

\begin{proof}
Since $\calO(D+D')\cong\calO(D)\otimes\calO(D')$, $\chi$ and $\chi'$ furnish a trivialization 
\[\chi\otimes\chi'\colon\calO(D+D')\cong\calO_X.
\]
As $|D+D'|=|D|\amalg|D'|$, we thus obtain isomorphisms
\begin{align*}
\wt\CH^1_{|D+D'|}(X,\calO(D+D'))&\cong\wt\CH^1_{|D|}(X,\calO(D+D'))\oplus\wt\CH^1_{|D'|}(X,\calO(D+D'))\\
&\cong\wt\CH^1_{|D|}(X)\oplus\wt\CH^1_{|D'|}(X)\\
&\cong\wt\CH^1_{|D|}(X,\calO(D))\oplus\wt\CH^1_{|D'|}(X,\calO(D')).
\end{align*}

To show the last claim, let $D$ and $D'$ be given by the data $\{(U_i,f_i)\}$ respectively $\{(U_i,f_i')\}$, so that $D+D'=\{(U_i,f_if_i')\}$. Let $x\in X^{(1)}\cap|D|$, and choose an $i$ such that $x\in U_i$. Since the vanishing loci of $f_i$ and $f_i'$ are disjoint we may assume that $f_i'\in\Gamma(U_i,\calO_X^\times)$, shrinking $U_i$ if necessary. Hence $\partial_x([f_i'])=0$, and we obtain
\begin{align*}
\partial_x([f_if_i']\otimes (f_if_i')^{-1}) &= \partial_x(([f_i']+\ip{f_i'}[f_i])\otimes (f_if_i')^{-1})\\
&=\ip{f_i'}\ip{(f_i')^{-1}}\partial_x([f_i]\otimes f_i^{-1})\\
&=\partial_x([f_i]\otimes f_i^{-1}).
\end{align*}
Thus $\partial_x([f_if_i']\otimes(f_if_i')^{-1})=\wt\ord_x(D)$. 
A similar argument shows that 
\[
\partial_x([f_if_i']\otimes(f_if_i')^{-1})=\wt\ord_x(D')
\] 
for all $x\in X^{(1)}\cap|D'|$, and the result follows.
\end{proof}

If we require a condition on the line bundle $\calO(D)$ and on the support of $D$, the class $\wt\Div(D)$ does indeed give rise to a finite $\rmMW$-correspondence:

\begin{lemma}\label{lemma:cartier}
Let $X$ and $Y$ be smooth connected $k$-schemes with $\dim Y=1$. Let $D$ be a Cartier divisor on $X\times Y$. Suppose that
\begin{itemize}
\item there is an isomorphism $\chi\colon\calO_{X\times Y}(D)\xrightarrow{\cong} \omega_Y$, and
\item each irreducible component of the support $|D|$ of $D$ is finite and surjective over $X$. 
\end{itemize}
Then the image of $\wt\Div(D)$ under the isomorphism 
\[\wt\CH^1_{|D|}(X\times Y, \calO_{X\times Y}(D))\xrightarrow{\cong}\wt\CH^1_{|D|}(X\times Y,\omega_Y)\] 
induced by $\chi$ defines a finite $\rmMW$-correspondence $\wt\Div(D,\chi)\in\wt\Cor_k(X,Y)$.
\end{lemma}

\begin{proof}
By assumption, $|D|$ is an admissible subset, hence the claim follows.
\end{proof}

\begin{lemma}\label{lemma:cartier-functoriality}
Assume the hypotheses of \Cref{lemma:cartier}, and let $f\colon X'\to X$ be a morphism of smooth $k$-schemes.
Then
\[
\wt\Div(D,\chi)\circ f=\wt\Div((f\times 1)^*D,(f\times1)^*\chi)\in\wt\Cor_k(X',Y).
\]
\end{lemma}

\begin{proof}
As $\wt\Div(D,\chi)\circ f=(f\times 1)^*\wt\Div(D,\chi)$, the claim follows from the fact that $(f\times1)^*$ commutes with the boundary map $\partial$ in the Rost--Schmid complex.
\end{proof}

For later reference, let us also state the version of \Cref{cor:supp-open} for Cartier-divisors:

\begin{lemma}\label{lemma:cartier-open}
Assume the hypotheses of \Cref{lemma:cartier}. Suppose moreover that $j\colon V\to Y$ is a Zariski open subscheme of $Y$ such that support $|D|$ is contained in $X\times V$. Then there exists a unique finite $\rmMW$-correspondence $\beta\in\wt\Cor_k(X,V)$ such that $j\circ\beta=\wt\Div(D,\chi)$. In fact, $\beta$ is given by
\[
\beta=\wt\Div((1\times j)^*D,(1\times j)^*\chi).
\]
\end{lemma}

\begin{proof}
By the same argument as in the proof of \Cref{lemma:cartier-functoriality} we have 
\[(1\times j)^*\wt\Div(D,\chi)=\wt\Div((1\times j)^*D,(1\times j)^*\chi).\]
Hence the claim follows from \Cref{cor:supp-open}.
\end{proof}

The above lemmas give a procedure to construct a morphism of pairs from a Cartier divisor:

\begin{lemma}\label{lemma:cartier-pair}
Assume the hypotheses of \Cref{lemma:cartier}, and let $j_X\colon U\to X$ and $j_Y\colon V\to Y$ be open subschemes. Let $D'\defeq D|_{U\times Y}$ be the restriction of $D$ to $U\times Y$. Suppose that $\vert D'\vert\subseteq U\times V$. Then 
\[
(\wt\Div(D,\chi),\wt\Div((j_X\times j_Y)^*D,(j_X\times j_Y)^*\chi))\in\wt\Cor_k^\pair((X,U),(Y,V)).
\]
\end{lemma}

\begin{proof}
By \Cref{lemma:cartier-functoriality}, $\wt\Div((j_X\times j_Y)^*D)=(j_X\times j_Y)^*\wt\Div(D)$, hence the claim follows from \Cref{lemma:making-pairs}.
\end{proof}

We will frequently make use of the following well known fact in order to determine if the support of a given principal divisor satisfies the hypotheses of \Cref{lemma:cartier}:

\begin{lemma}\label{lemma:monic-fin-surj}
Let $A$ be a ring, and suppose that $P$ is a monic polynomial in $A[t]$. Then $\Spec(A[t]/(P))\to\Spec(A)$ is finite, and every irreducible component of $\Spec(A[t]/(P))$ surjects onto $\Spec(A)$.
\end{lemma}

\begin{proof}
Write $P(t)=t^n+a_{n-1}t^{n-1}+\cdots+a_0$, and let $M\defeq A[t]/(P)$. Then $M$ is generated as an $A$-module by $1,t,\dots,t^{n-1}$, hence $\Spec(A[t]/(P))\to\Spec(A)$ is finite. As $A[t]/(P)$ is integral over $A$, it follows that the morphism is surjective as well.
\end{proof}

\section{Zariski excision on the affine line}\label{section:Zar-excision}
The aim of this section is to prove the following excision result:

\begin{theorem}\label{thm:Zar-excision}
Let $x\in\A^1$ be a closed point and suppose that $V\subseteq U\subseteq\A^1$ are two Zariski open neighborhoods of $x$. Let $i\colon V\hookrightarrow U$ denote the inclusion, and let $\scrF\in\wt\PSh(k)$ be a homotopy invariant presheaf with $\rmMW$-transfers. Then the induced map
\[i^*\colon \dfrac{\scrF(U\setminus x)}{\scrF(U)}\to \dfrac{\scrF(V\setminus x)}{\scrF(V)}\]
is an isomorphism.
\end{theorem}

The proof of Zariski excision proceeds in three steps. First we prove:

\begin{theorem}[Injectivity on the affine line]\label{thm:inj-aff-line}
With the notation in \Cref{thm:Zar-excision}, there exists a finite $\rmMW$-correspondence $\Phi\in\wt\Cor_k(U,V)$ such that
\[
\ol i\circ\ol \Phi=\id_U
\]in $\hwtCor_k$.
\end{theorem}

\Cref{thm:inj-aff-line} implies that $\Phi^*\circ i^*=\id_{\scrF(U)}$ for any homotopy invariant $\scrF\in\wt\PSh(k)$, i.e., that $i^*$ is injective. Letting $V=U\setminus x$, this means that $\scrF(U)$ is a subgroup of $\scrF(U\setminus x)$, justifying the notation of \Cref{thm:Zar-excision}. 

The next step is then to extend \Cref{thm:inj-aff-line} to the category $\wt\Cor_k^\pair$ of pairs. By abuse of notation, write $i$ also for the inclusion $i\colon(V,V\setminus x)\hookrightarrow(U,U\setminus x)$ in $\SmOp_k$. By \Cref{prop:construction}, $i$ induces a map 
\[
i^*\colon \dfrac{\scrF(U\setminus x)}{\scrF(U)}\to \dfrac{\scrF(V\setminus x)}{\scrF(V)}
\]
on the quotient, and the following theorem tells us that $i^*$ is injective:

\begin{theorem}[Injectivity of Zariski excision]\label{thm:Zar-inj}
There exists a finite $\rmMW$-correspondence $\Phi\in\wt\Cor_k^\pair((U,U\setminus x),(V,V\setminus x))$ such that  
\[
\ol i\circ\ol\Phi=\id_{(U,U\setminus x)}
\]in $\hwtCor_k^\pair$.
\end{theorem}

In the final step we establish surjectivity of $i^*$:

\begin{theorem}[Surjectivity of Zariski excision]\label{thm:Zar-surj}
With the notation in \Cref{thm:Zar-excision}, there exist finite $\rmMW$-correspondences $\Psi\in\wt\Cor_k^\pair((U,U\setminus x),(V,V\setminus x))$ and $\Theta\in\wt\Cor_k^\pair((V,V\setminus x),(V\setminus x,V\setminus x))$ such that
\[
\ol\Psi\circ\ol i-\ol{j_V}\circ\ol\Theta=\id_{(V,V\setminus x)}
\]in $\hwtCor_k^\pair$, where $j_V\colon (V\setminus x, V\setminus x)\hookrightarrow (V,V\setminus x)$ denotes the inclusion in $\SmOp_k$.
\end{theorem}

We note that \Cref{thm:Zar-excision} is a consequence of Theorems \ref{thm:Zar-inj} and \ref{thm:Zar-surj}:

\begin{proof}[Proof of \Cref{thm:Zar-excision}]
As $\Phi$ is a morphism of pairs by \Cref{thm:Zar-inj}, \Cref{prop:construction} tells us that $\Phi$ induces a morphism on the quotient 
\[
\Phi^*\colon\dfrac{\scrF(V\setminus x)}{\scrF(V)}\to\dfrac{\scrF(U\setminus x)}{\scrF(U)}.\] 
Moreover, $\Phi^*\circ i^*=\id$ by \Cref{thm:Zar-inj}, hence $i^*$ is injective.

On the other hand, as $\Theta$ maps to $(V\setminus x,V\setminus x)$ by \Cref{thm:Zar-surj}, it follows that $j_V\circ\Theta$ induces the trivial map on the quotient. Hence 
\[i^*\circ\Psi^*=\id\colon\dfrac{\scrF(V\setminus x)}{\scrF(V)}\to\dfrac{\scrF(V\setminus x)}{\scrF(V)},\] 
so that $i^*$ is surjective.
\end{proof}

It is therefore enough to prove Theorems \ref{thm:inj-aff-line}, \ref{thm:Zar-inj}, and \ref{thm:Zar-surj}.

\section{Injectivity on the affine line}\label{section:inj-aff-line}
We continue with the same notation as in \Cref{thm:Zar-excision}. Thus $V\subseteq U\subseteq\A^1$ are two Zariski open neighborhoods of a closed point $x\in\A^1$, with inclusion $i\colon V\to U$. In order to produce the desired $\rmMW$-correspondence $\Phi\in\wt\Cor_k(U,V)$ of \Cref{thm:inj-aff-line}, we will need to consider certain ``thick diagonals'' $\Delta_m\in\wt\Cor_k(U,U)$, constructed as follows.

Let $U\times U\subseteq\A^2$ have coordinates $X$ and $Y$, respectively, and let $\Delta\defeq\Delta(U)\subseteq U\times U$ denote the diagonal. For each $m\ge1$, let $f_m$ denote the polynomial $f_m(X,Y)\defeq (Y-X)^m\in k[U\times U]$. As $f_m$ is monic in $Y$, it follows from \Cref{lemma:monic-fin-surj} that the support of the divisor 
\[
D_m\defeq\calV(f_m)\defeq\{f_m=0\}\subseteq U\times U
\]
is finite and surjective over $U$. Moreover, as $D_m$ is a principal Cartier divisor on $U\times U$, there is a trivialization $\calO(D_m)\cong \calO_{U\times U}$ given by $f_m^{-1}\mapsto1$. We further obtain an isomorphism $\chi_m\colon\calO(D_m)\cong\omega_U$ by $f_m^{-1}\mapsto dY$. By \Cref{lemma:cartier}, it follows that the divisor $D_m$ gives rise to a finite $\rmMW$-correspondence from $U$ to $U$.

\begin{definition}\label{def:thick-diagonal}
For each $m\ge1$, let $\Delta_m\defeq\wt\Div(D_m,\chi_m)\in\wt\Cor_k(U,U)$ be the finite $\rmMW$-correspondence defined by the data $D_m$ and $\chi_m$ above.
\end{definition}

\begin{remark}
By the definition of $\wt\Div(D_m,\chi_m)$, we see that $\Delta_m$ is given by the total residue 
\[
\Delta_m=\partial([f_m]\otimes dY)\in \wt\CH^1_{\Delta}(U\times U,\omega_{U})
\]of the element $[f_m]\otimes dY\in \rmK_1^{\rmMW}(k(U\times U),\omega_U)$. Thus the support of the $\rmMW$-correspondence $\Delta_m$ is the diagonal $\Delta=D_1\subseteq U\times U$.
\end{remark}

\begin{lemma}\label{lemma:difference}
For any $m\ge0$ we have $\Delta_{m+1}-\Delta_{m}=\ip{-1}^m\cdot\Delta_1\in\wt\Cor_k(U,U)$, with $\Delta_1=\id_U$. 
\end{lemma}

\begin{proof}
Since $\Delta_m$ is supported on the diagonal $\Delta\subseteq U\times U$, it suffices to compute the residue $\partial_y([f_m]\otimes dY)$ at the codimension $1$-point $y\in (U\times U)^{(1)}$ corresponding to the diagonal. 

Recall from \cite[Lemma 3.14]{Morel-over-a-field} that for any integer $n\ge0$ we have $[a^n]=n_\epsilon[a]$ in $\rmK_1^{\rmMW}$, where $n_\epsilon=\sum_{i=1}^{n}\ip{(-1)^{i-1}}$. We thus get
\[
\partial_y([f_m]\otimes dY)=m_\epsilon\otimes\ol{(Y-X)}dY\in \rmK^{\rmMW}_0(k(y),(\frakm_y/\frakm_y^2)^\vee\otimes(\omega_{U})_y).
\]
For $m=1$, this reads $\Delta_1=\ip1\otimes\ol{(Y-X)}dY=\id_U$. In the general case we obtain
\[
\Delta_{m+1}-\Delta_{m}=((m+1)_\epsilon-m_\epsilon)\otimes\ol{(Y-X)}dY=\ip{(-1)^m}\cdot\id_U,
\]using that $\Delta_1=\id_U\in\wt\Cor_k(U,U)$.\end{proof}

Our next objective is to prove the following:

\begin{lemma}
For $m\gg0$ there exists a finite $\rmMW$-correspondence $\Phi_m\colon U\to V$ such that $i\circ\Phi_m=\Delta_m$ in $\hwtCor_k(U,U)$. 
\end{lemma}

Having established these properties of $\Delta_m$ and $\Phi_m$, we will set $\Phi\defeq\Phi_{m+1}-\Phi_m$ and show that we then have $i\circ\Phi\sim_{\A^1}\id_U$ provided $m$ is an even integer $\gg0$. To define $\Phi_m$, we will need to ensure the existence of polynomials with certain specified properties.

\begin{lemma}[\protect{\cite[§5]{hty-inv}}]\label{lemma:CRT1}
Let $A\defeq\A^1\setminus U$ and $B\defeq U\setminus V$. For $m\gg0$, there exists a polynomial $G_m\in k[U][Y]=k[U\times \A^1]$, monic and of degree $m$ in $Y$, satisfying the following properties:
\begin{itemize}
\item[$(1)$]$G_m(Y)|_{U\times B}=1.$
\item[$(2)$]$G_m(Y)|_{U\times A}=(Y-X)^m|_{U\times A}$. 
\item[$(3)$] $G_m(Y)|_{U\times x}=(Y-X)^m|_{U\times x}$.
\end{itemize}
\end{lemma}

\begin{remark}
The above polynomials, as well as those in Sections \ref{section:Zar-inj} and \ref{section:Zar-surj}, are all constructed using variants of the Chinese remainder theorem, allowing us to find polynomials with specified behavior at given subschemes. The requirement that the desired polynomial be monic can be thought of as specifying its behavior at infinity. For example, the Chinese remainder theorem establishes a surjection $k[U\times\A^1]\to k[U\times A]\oplus k[U\times B]$, from which we can deduce \Cref{lemma:CRT1}.
\end{remark}

\begin{lemma}
Let $D_{G_m}$ be the divisor on $U\times U$ defined by $G_m$, and let $\phi_m\colon\calO(D_{G_m})\cong \omega_U$ be the isomorphism determined by choosing the generator $dY$ for $\omega_U$. Then 
\[\wt\Div((1\times i)^*D_{G_m},(1\times i)^*\phi_m)\in\wt\Cor_k(U,V).\]
\end{lemma}

\begin{proof}
Since $G_m$ is monic in $Y$,  the support $\calV(G_m)$ of $D_{G_m}$ is finite and surjective over $U$ by \Cref{lemma:monic-fin-surj}. Using the trivializations of $\calO(D_{G_m})$ and of $\omega_U$, \Cref{lemma:cartier} implies that $\wt\Div(D_{G_m},\phi_m)\in\wt\Cor_k(U,U)$. Now, the polynomial $G_m$ satisfies the following:
\begin{itemize}
\item $G_m|_{U\times A}\in k[U\times A]^\times$. This follows from the fact that $U\times A=U\times(\A^1\setminus U)$ contains no diagonal points.
\item $G_m|_{U\times B}\in k[U\times B]^\times$. This is obvious, as $G_m|_{U\times B}=1$.
\end{itemize}
The above properties imply that $\calV(G_m)\subseteq U\times V$. Hence the claim follows from \Cref{lemma:cartier-open}.  
\end{proof}

\begin{definition}
For $m\gg0$, we define $\Phi_m\defeq\wt\Div\p*{(1\times i)^*D_{G_m},(1\times i)^*\phi_m}\in\wt\Cor_k(U,V)$.
\end{definition}

We now aim to define a homotopy $\scrH_m\colon i\circ \Phi_m\sim_{\A^1}\Delta_m$. Consider the product $\A^1\times U\times\A^1$, where $\theta$ is the coordinate of the first copy of $\A^1$, $U$ has coordinate $X$ and the last $\A^1$ has coordinate $Y$. Let $H_\theta\in k[\A^1\times U\times\A^1]$ be the polynomial
\[
H_\theta(Y)\defeq \theta G_m+(1-\theta)(Y-X)^m.
\]
Since $U\times A$ contains no diagonal points, the restriction $G_m(Y)|_{U\times A}=(Y-X)^m|_{U\times A}$ does not vanish on $U\times A$. It follows that 
\[
H_\theta(Y)|_{\A^1\times U\times A}=(Y-X)^m|_{\A^1\times U\times A}\in k[\A^1\times U\times A]^\times.
\]
Hence $\calV(H_\theta)\subseteq\A^1\times U\times U$. Let $D_{H_\theta}$ be the principal Cartier divisor on $\A^1\times U\times U$ defined by $H_\theta$, and let $\psi\colon\calO(D_{H_\theta})\cong \omega_U$ be the isomorphism given by choosing the generator $dY$ for $\omega_U$.

\begin{lemma}
Let $\scrH_m\defeq \wt\Div(D_{H_\theta},\psi)$. Then $\scrH_m\in\wt\Cor_k(\A^1\times U,U)$. 
\end{lemma}

\begin{proof}
As $G_m$ is monic and of degree $m$ in $Y$, it follows that the linear combination $H_\theta$ of $G_m$ and $(Y-X)^m$ is also monic and of degree $m$ in $Y$. Therefore the support $\calV(H_\theta)$ of $D_{H_\theta}$ is finite and surjective over $\A^1\times U$ by \Cref{lemma:monic-fin-surj}. The result then follows from \Cref{lemma:cartier}. 
\end{proof}

\begin{lemma}\label{lemma:homotopy-start-end}
Let $\scrH_m|_0\defeq\scrH_m\circ i_0,\scrH_m|_1\defeq\scrH_m\circ i_1\in\wt\Cor_k(U,U)$ denote the respective precompositions of $\scrH_m\in\wt\Cor_k(\A^1\times U,U)$ with the rational points $i_0,i_1\colon U\to \A^1\times U$. Then $\scrH_m|_0=\Delta_m$ and $\scrH_m|_1=i\circ\Phi_m$.
\end{lemma}

\begin{proof}
By \Cref{lemma:cartier-functoriality} we have 
\[
\scrH_0=\wt\Div((i_0\times1)^*D_{H_\theta},(i_0\times1)^*\psi)=\wt\Div(D_m,\chi_m)=\Delta_m.
\]
On the other hand,
\[
\scrH_1=\wt\Div((i_1\times1)^*D_{H_\theta},(i_0\times1)^*\psi)=\wt\Div(D_{G_m},\phi_m)=i\circ\Phi_m
\]by \Cref{lemma:cartier-open}.
\end{proof}

We are now ready to prove the injectivity of the induced morphism $i^*\colon \scrF(U)\to \scrF(V)$, for any homotopy invariant $\scrF\in\wt\PSh(k)$.

\begin{proof}[Proof of \Cref{thm:inj-aff-line}]
Let $m\gg0$ be an integer large enough so that the polynomial $G_m$ of \Cref{lemma:CRT1} exists. If $\Phi\defeq\Phi_{2m+1}-\Phi_{2m}$, we then have $i\circ\Phi\sim_{\A^1}(\Delta_{2m+1}-\Delta_{2m})=\ip{(-1)^{2m}}\id_U=\id_U$ by \Cref{lemma:difference}. As $\scrF$ is homotopy invariant, this yields $\Phi^*\circ i^*=\id_{\scrF(U)}$, hence $i^*$ is injective.
\end{proof}

\section{Injectivity of Zariski excision}\label{section:Zar-inj}
We wish to extend \Cref{thm:inj-aff-line} to the category of pairs $\wt\Cor_k^\pair$---in other words to produce a morphism $(\Phi_m,\Phi_m^x)\in\wt\Cor_k^\pair((U,U\setminus x),(V,V\setminus x))$ and a homotopy $(\scrH_m,\scrH_m^x)\in\wt\Cor_k^\pair(\A^1\times (U,U\setminus x),(U,U\setminus x))$ from $\Delta_m$ to $(i,i|_{V\setminus x})\circ(\Phi_m,\Phi_m^x)$. This establishes \Cref{thm:Zar-inj}. 

Let $j_U$ and $j_V$ denote the respective open immersions $j_U\colon U\setminus x\to U$ and $j_V\colon V\setminus x\to V$.

\begin{lemma}\label{lemma:pair}
Let $\Phi_m^x\defeq \wt\Div((j_U\times j_V)^*D_{G_m},(j_U\times j_V)^*\phi_m)$. Then $(\Phi_m,\Phi_m^x)$ constitutes a morphism in $\wt\Cor_k^\pair$ from $(U,U\setminus x)$ to $(V,V\setminus x)$.
\end{lemma}

\begin{proof}
By \Cref{lemma:cartier-pair}, it suffices to show that the support of $(j_U\times 1)^*D_{G_m}$ is contained in $(U\setminus x)\times(V\setminus x)$. As we already know that $\calV(G_m)\cap((U\setminus x)\times\A^1)\subseteq (U\setminus x)\times V$, it is enough to check that $G_m$ does not vanish on $(U\setminus x)\times x$. By condition $(3)$ of \Cref{lemma:CRT1}, $G_m(Y)|_{U\times x}=(Y-X)^m|_{U\times x}$. As $(U\setminus x)\times x$ contains no diagonal points, it therefore follows that $G_m|_{(U\setminus x)\times x}\in k[(U\setminus x)\times x]^\times$. Hence $\calV(G_m)\cap((U\setminus x)\times\A^1)\subseteq(U\setminus x)\times(V\setminus x)$.
\end{proof}

\begin{lemma}\label{lemma:image}
Let $\scrH_\theta^x\defeq\wt\Div(((1\times j_U)\times j_U)^*D_{H_\theta},((1\times j_U)\times j_U)^*\psi)$. Then
\[
(\scrH_\theta,\scrH_\theta^x)\in\wt\Cor_k^\pair(\A^1\times (U,U\setminus x),(U,U\setminus x)).
\]
\end{lemma}

\begin{proof}
In light of \Cref{lemma:cartier-pair}, it remains to check that 
\[
\calV(H_\theta)\cap(\A^1\times(U\setminus x)\times\A^1)\subseteq\A^1\times(U\setminus x)\times(U\setminus x).
\]
It is sufficient to show that $H_\theta$ does not vanish on $\A^1\times(U\setminus x)\times x$. But
\[
H_\theta(Y)|_{\A^1\times(U\setminus x)\times x}=\theta\cdot (Y-X)^m+(1-\theta)\cdot(Y-X)^m=(Y-X)^m|_{\A^1\times(U\setminus x)\times x},
\]
and $(Y-X)^m|_{\A^1\times(U\setminus x)\times x}\in k[\A^1\times(U\setminus x)\times x]^\times$ as $(U\setminus x)\times x$ contains no diagonal points. Whence the claim.
\end{proof}

\begin{proof}[Proof of \Cref{thm:Zar-inj}]
By a similar argument as in the proof of \Cref{lemma:homotopy-start-end}, $(\scrH_\theta,\scrH_\theta^x)$ is a homotopy from $\Delta_m$ to $(i,i|_{V\setminus x})\circ(\Phi_m,\Phi_m^x)$. Thus the same proof as that of \Cref{thm:inj-aff-line} applies.
\end{proof}

\section{Surjectivity of Zariski excision}\label{section:Zar-surj}
We proceed to prove \Cref{thm:Zar-surj}. To begin with, we interpolate polynomials in a similar fashion as \Cref{lemma:CRT1}:

\begin{lemma}[\protect{\cite[§5]{hty-inv}}]\label{lemma:CRT2}
For $m\gg0$ there exists a polynomial $G_m(Y)\in k[U][Y]=k[U\times\A^1]$, monic and of degree $m$ in $Y$, satisfying the following properties:
\begin{itemize}
\item[$(1')$]$G_m(Y)|_{U\times B}=1.$
\item[$(2')$]$G_m(Y)|_{U\times A}=(Y-X)|_{U\times A}$. 
\item[$(3')$] $G_m(Y)|_{U\times x}=(Y-X)|_{U\times x}$.
\end{itemize}
\end{lemma}

\begin{lemma}[\protect{\cite[§5]{hty-inv}}]\label{lemma:CRT3}
For $m\gg0$ there exists a polynomial $F_{m-1}(Y)\in k[V][Y]=k[V\times\A^1]$, monic and of degree $m-1$ in $Y$, satisfying the following properties:
\begin{itemize}
\item[$(1'')$]$F_{m-1}(Y)|_{V\times B}=(Y-X)^{-1}\in k[V\times B]^\times.$
\item[$(2'')$]$F_{m-1}(Y)|_{V\times A}=1$. 
\item[$(3'')$]$F_{m-1}(Y)|_{\Delta(V)}=1$.
\end{itemize}
\end{lemma}

\begin{remark}
As $B=U\setminus V$, the set $V\times B$ does not contain any diagonal points. Hence the function $Y-X$ is invertible on $V\times B$, so $(1'')$ makes sense.
\end{remark}

\begin{definition}
Set $E_m\defeq (Y-X)\cdot F_{m-1}\in k[V][Y]$ and $H_\theta\defeq \theta G_m+(1-\theta)E_m\in k[\A^1\times V][Y]$, where $\theta$ is the coordinate of $\A^1$.
\end{definition}
Observe that the divisor $\calV(E_m)$ satisfies $\calV(E_m)=\calV(Y-X)\cup \calV(F_{m-1})=\Delta(V)\cup\calV(F_{m-1})$. In fact, by $(3'')$, this union is a disjoint union. Moreover, using the definition of $F_{m-1}$ we see that $E_m$ enjoys the following properties:
\begin{itemize}
\item[$(1_E)$]$E_m(Y)|_{V\times B}=1=G_m(Y)|_{V\times B}$.
\item[$(2_E)$]$E_m(Y)|_{V\times A}=(Y-X)|_{V\times A}=G_m(Y)|_{V\times A}$.
\item[$(3_E)$]$E_m(Y)|_{V\times x}=(Y-X)|_{V\times x}=G_m(Y)|_{V\times x}$.
\end{itemize}
The last property $(3_E)$ implies:
\begin{itemize}
\item[$(3_E')$]$E_m(Y)|_{(V\setminus x)\times x}=G_m(Y)|_{(V\setminus x)\times x}\in k[(V\setminus x)\times x]^\times$.
\end{itemize}

Let us first construct the finite $\rmMW$-correspondence $\Psi\in\wt\Cor_k(U,V)$ using the polynomial $G_m$ of \Cref{lemma:CRT2} for $m\gg0$. By \Cref{lemma:CRT2}, $\calV(G_m)\subseteq U\times V$, and we may consider the principal divisor $D_{G_m}$ on $U\times V$ defined by $G_m$. Let $\psi\colon\calO(D_{G_m})\cong\omega_V$ be the isomorphism determined by choosing the generator $dY$ for $\omega_V$.

\begin{lemma}
Put $\Psi\defeq\wt\Div(D_{G_m},\psi)$ and $\Psi^x\defeq\wt\Div((j_U\times j_V)^*D_{G_m},(j_U\times j_V)^*\psi)$.
Then
\[
(\Psi,\Psi^x)\in\wt\Cor_k^\pair((U,U\setminus x),(V,V\setminus x)).
\]
\end{lemma}

\begin{proof}
Since $G_m$ is monic in $Y$, $\calV(G_m)$ is finite and surjective over $U$ by \Cref{lemma:monic-fin-surj}. Thus \Cref{lemma:cartier} ensures that $\Psi$ is a finite $\rmMW$-correspondence from $U$ to $V$. Moreover, as $G_m(Y)|_{U\times x}=(Y-X)|_{U\times x}$, it follows that $G_m|_{(U\setminus x)\times x}$ is invertible on $(U\setminus x)\times x$. Hence there is an inclusion 
\[\calV(G_m)\cap((U\setminus x)\times V)\subseteq(U\setminus x)\times(V\setminus x).\]
By \Cref{lemma:cartier-pair} it follows that $(\Psi,\Psi^x)$ is a morphism of pairs from $(U,U\setminus x)$ to $(V,V\setminus x)$.
\end{proof}

In order to define the desired homotopy, we proceed in a familiar fashion. By $(1_E)$ and $(2_E)$, $H_\theta$ is invertible on $\A^1\times V\times B$ and $\A^1\times V\times A$. Hence $\calV(H_\theta)\subseteq\A^1\times V\times V$, and we may consider the divisor $D_{H_\theta}$ on $\A^1\times V\times V$. We let $\chi\colon\calO(D_{H_\theta})\cong\omega_V$ be the isomorphism given by choosing the generator $dY$ for $\omega_V$.

\begin{lemma}
Let $\scrH_\theta\defeq\wt\Div(D_{H_\theta},\chi)$ and 
\[\scrH_\theta^x\defeq\wt\Div(((1\times j_V)\times j_V)^*D_{H_\theta},((1\times j_V)\times j_V)^*\chi).\]
Then
\[
(\scrH_\theta,\scrH_\theta^x)\in\wt\Cor_k^\pair(\A^1\times (V,V\setminus x),(V,V\setminus x)).
\]
\end{lemma}

\begin{proof}
To see that $\scrH_\theta$ is a finite $\rmMW$-correspondence from $\A^1\times V$ to $V$, note that both $G_m$ and $E_m$ are monic and of the same degree in $Y$. Therefore the linear combination $H_\theta$ of $G_m$ and $E_m$ is also monic in $Y$, and it follows that the support $\calV(H_\theta)$ of $D_{H_\theta}$ is finite and surjective over $\A^1\times V$ by \Cref{lemma:monic-fin-surj}. Hence $\scrH_\theta\in\wt\Cor_k(\A^1\times V,V)$ by \Cref{lemma:cartier}.

Turning to $\scrH_\theta^x$, we must show that 
\[\calV(H_\theta)\cap(\A^1\times(V\setminus x)\times V)\subseteq \A^1\times(V\setminus x)\times(V\setminus x).\]
We already know that $H_\theta$ is invertible on $\A^1\times (V\setminus x)\times A$ and on $\A^1\times(V\setminus x)\times B$. It remains to check the set $\A^1\times(V\setminus x)\times x$. But by $(3_E)$ and $(3_E')$ we have
\[
E_m(Y)|_{(V\setminus x)\times x}=G_m(Y)|_{(V\setminus x)\times x}=(Y-X)|_{(V\setminus x)\times x},
\]
which is invertible as $(V\setminus x)\times x$ does not intersect the diagonal. Therefore the linear combination $H_\theta$ of $E_m$ and $G_m$ is also invertible on $(V\setminus x)\times x$, and the claim follows. Using \Cref{lemma:cartier-pair}, this shows that $(\scrH_\theta,\scrH_\theta^x)$ constitutes a morphism of pairs from $\A^1\times(V,V\setminus x)$ to $(V,V\setminus x)$.
\end{proof}

Let us compute the start-, and endpoints $\scrH_0$, $\scrH_1$ of the homotopy $\scrH_\theta$---that is, the precomposition of $\scrH_\theta$ with the rational points $i_0,i_1\colon V\to \A^1\times V$.

\begin{lemma}\label{lemma:hty01}
We have $\scrH_0=\id_V+j_V\circ\Theta$ where $\Theta\in\wt\Cor_k(V,V\setminus x)$. On the other hand, $\scrH_1=\Psi\circ i$, where $i\colon V\hookrightarrow U$ is the inclusion.
\end{lemma}

\begin{proof}
By \Cref{lemma:cartier-functoriality} we have $\scrH_1=\wt\Div((i_1\times1)^*D_{H_\theta},(i_1\times1)^*\chi)=\Psi\circ i$. As for $\scrH_0$, we have
\[
\scrH_0=\wt\Div((i_0\times1)^*D_{H_\theta},(i_0\times1)^*\chi)=\wt\Div(D_{E_m},(i_0\times1)^*\chi),
\]
where $D_{E_m}$ is the principal Cartier divisor on $V\times V$ defined by the polynomial $E_m$. Let $D_{F_{m-1}}$ be the principal divisor on $V\times V$ defined by $F_{m-1}$. As $\calV(E_m)=\Delta(V)\amalg\calV(F_{m-1})$, \Cref{lemma:cartier-sum} tells us that 
\[\scrH_0=\Delta_1+\wt\Div(D_{F_{m-1}},(i_0\times1)^*\chi).\] 
Here $\Delta_1$ is the divisor defined in \Cref{def:thick-diagonal}, satisfying $\Delta_1=\id_V$. As $\calV(F_{m-1})\subseteq V\times(V\setminus x)$, \Cref{lemma:cartier-open} ensures that there is a unique element \[\Theta\in\wt\CH^1_{\calV(F_{m-1})}(V\times(V\setminus x),\omega_{V\setminus x})\] such that $j_V\circ\Theta =\wt\Div(D_{E_m},(i_0\times1)^*\chi)$. By \Cref{lemma:CRT3}, $\calV(F_{m-1})$ is finite and surjective over $V\setminus x$, and hence $\Theta\in\wt\Cor_k(V,V\setminus x)$.
\end{proof}

\begin{proof}[Proof of \Cref{thm:Zar-surj}]
The content of \Cref{thm:Zar-surj} is a rephrasing of \Cref{lemma:hty01}.
\end{proof}

\section{Zariski excision on \texorpdfstring{$\A^1_K$}{A1K}}\label{section:A1K}
We now aim to extend the results of \Cref{section:Zar-excision} to open subsets of $\A^1_K$, where $K=k(X)$ is the function field of some integral $k$-scheme $X\in\Sm_k$. This can be achieved by the following trick, which was suggested to the author by I. Panin: given a presheaf $\scrF\in\wt\PSh(k)$, we can extend $\scrF$ to a presheaf $\scrF^X$ on a certain full subcategory of $\wt\Cor_{K}$, and then use Zariski excision for presheaves on $\wt\Cor_K$. 

In this section, the field $k$ is assumed to be of characteristic $0$.

\begin{remark}\label{rmk:ground-field}
Notice that the results of \Cref{section:Zar-excision} show that Zariski excision on $\A^1_K$ holds for any homotopy invariant presheaf on $\wt\Cor_K$ by simply letting the ground field be $K$. The point of this section, however, is to show that we can obtain Zariski excision on $\A^1_K$ also for homotopy invariant presheaves on $\wt\Cor_k$. 
\end{remark}

\begin{definition}
Let $X\in\Sm_k$ be a smooth integral $k$-scheme, and let $K\defeq k(X)$ be the function field of $X$. We define the category $\wt\Cor_K^X$ as follows. Its objects are are pairs $(Y,V\subseteq Y_K)$ consisting of a smooth $k$-scheme $Y\in\Sm_k$ along with an open subscheme $V$ of $Y_K\defeq Y\times_k\Spec(k(X))$. The morphisms of $\wt\Cor_K^X$ are given as
\[
\Hom_{\wt\Cor_K^X}((Y,V),(Y',V'))\defeq\wt\Cor_K(V,V').
\]
Abusing notation, we may write simply $V$ for an object $(Y,V)$ of $\wt\Cor_K^X$.
\end{definition}

\begin{remark}\label{rem:Xhtpy}
Since any open subscheme $V$ of $Y_K$ is $K$-smooth, $\wt\Cor_K^X$ is equivalent to the full subcategory of $\wt\Cor_K$ whose objects are those $V\in\wt\Cor_K$ for which there exists $Y\in\Sm_k$ along with an open embedding $V\hookrightarrow Y_K$.
\end{remark}

Let us fix some notation:

\begin{definition}
If $(Y,V)\in\wt\Cor_K^X$, we define the following subschemes of $Y_K$ and $Y\times_k X$:
\begin{itemize}
\item $Z\defeq Y_K\setminus V$;
\item $\scrZ\defeq\ol Z$, the Zariski closure of $Z$ in $Y\times_k X$;
\item $\scrV\defeq(Y\times_k X)\setminus\scrZ$.
\end{itemize}
Let also $\scrV_K\defeq\scrV\times_{X}\Spec(K)$ denote the generic fiber of the projection $p_X|_{\scrV}\colon \scrV\to X$. Note that we then have $\scrV_K=V$. Furthermore, for each open subscheme $X_i$ of $X$, set
\[
\scrV(X_i)\defeq\scrV\cap(Y\times_k X_i),
\]
the intersection being taken in $Y\times_k X$. Then we have $V=\scrV_K=\varprojlim_{X_i\subseteq X}\scrV(X_i)$, where the limit runs over all nonempty open subsets of $X$. In particular, $\scrV=\scrV(X)$. 
\end{definition}

\begin{definition}\label{def:scr-notation}
Let $\scrF\in\wt\PSh(k)$ be a presheaf. For any $V\in\wt\Cor_K^X$, we set
\[
\scrF^X(V)\defeq\varinjlim_i\scrF(\scrV(X_i)).
\]
In particular, if $(Y,V)=(\A^1_k,\A^1_K)$, then $\scrF^X(V)=\scrF(\A^1_K)=\varinjlim_i\scrF(\A^1_k\times_k X_i)$.
\end{definition}

\begin{remark}\label{rem:Xhtpy2}
Notice that if $\scrF\in\wt\PSh(k)$ is homotopy invariant, then $\scrF^X(\A^1_K)\cong\scrF^X(K)$.
\end{remark}

Our goal is now to promote $\scrF^X$ to a presheaf on $\wt\Cor_K^X$. For $U,V\in\wt\Cor_K^X$, this means that we need to define a natural restriction map $\alpha^*\colon\scrF^X(U)\to\scrF^X(V)$ for any $\alpha\in\wt\Cor_K^X(V,U)$. To do this we need some preparations. First, recall that we can write $U=\varprojlim_i\scrU(X_i)$, $V=\varprojlim_i\scrV(X_i)$, where $\scrU$, $\scrV$ and $X_i$ are as in \Cref{def:scr-notation}.

\begin{lemma}\label{lemma:fundamental-trick}
With the notations as above, we have a natural isomorphism
\[
\varinjlim_i\wt\Cor_{X_i}(\scrV(X_i),\scrU(X_i))\xrightarrow{\cong}\wt\Cor_K(V,U).
\]
\end{lemma}

\begin{proof}
Rewriting $U$ as $\scrU\times_X\Spec(K)$, we obtain the chain of natural isomorphisms
\begin{align*}
\wt\Cor_K(V,U) &\cong\wt\Cor_K(V,\scrU\times_X\Spec(K))\\
&\cong\wt\Cor_X(V,\scrU)\\
&\cong\varinjlim_i\wt\Cor_X(\scrV(X_i),\scrU)\\
&\cong\varinjlim_i\wt\Cor_{X_i}(\scrV(X_i),\scrU(X_i)).
\end{align*}
Here the penultimate isomorphism follows from a similar argument as that of \cite[Lemmas 4.6 and 5.10]{Calmes-Fasel}.
\end{proof}

For any $\alpha\in\wt\Cor_K^X(V,U)$, we can now define a natural map
\[
\alpha^*\colon\scrF^X(U)\to\scrF^X(V)
\]
as follows. Using \Cref{lemma:fundamental-trick}, we may choose a representative $\alpha_i\in\wt\Cor_{X_i}(\scrV(X_i),\scrU(X_i))$ mapping to $\alpha$. Let 
\[
f_i\colon\scrV(X_i)\times_{X_i}\scrU(X_i)\to\scrV(X_i)\times_k\scrU(X_i)
\]
denote the canonical morphism. By \Cref{lemma:rel-MW-to-MW}, $f_i$ induces a homomorphism 
\[
(f_i)_*\colon\wt\Cor_{X_i}(\scrV(X_i),\scrU(X_i))\to\wt\Cor_k(\scrV(X_i),\scrU(X_i)).
\]

\begin{definition} 
With the notations as above, set
\[
\alpha^*\defeq\varinjlim_{j\ge i}\p*{(f_j)_*(\alpha_j)}^*\colon\scrF^X(U)\to\scrF^X(V).
\]
\end{definition}

For any pair of indices $i\le j$, the following commutative diagram in $\wt\Cor_k$,
\[\begin{tikzcd}
\scrV(X_j)\ar[r,hook]\ar{d}[swap]{(f_j)_*(\alpha_j)} & \scrV(X_i)\ar{d}{(f_i)_*(\alpha_i)}\\
\scrU(X_j)\ar[r,hook] & \scrU(X_i),
\end{tikzcd}\]
shows that the definition of $\alpha^*$ does not depend on the lift $\alpha_i$.

\begin{lemma}[Injectivity on $\A^1_K$]\label{lemma:inj-A1K}
Let $(\A^1_k,U)$ and $(\A^1_k,V)$ be two objects of $\wt\Cor_K^X$ such that $V$ is nonempty and $V\subseteq U$. Write $i\colon V\hookrightarrow U$ for the inclusion. Then the induced map
\[
i^*\colon\scrF^X(U)\to\scrF^X(V)
\]
is injective for any homotopy invariant presheaf $\scrF\in\wt\PSh(k)$. 
\end{lemma}

\begin{proof}
Injectivity on the affine line gives a homotopy $\Phi\in\wt\Cor_K(U,V)$ such that $i\circ\Phi\sim_{\A^1}\id_U$. Since $\Phi$ is a morphism in $\wt\Cor_K^X$ and $\scrF^X$ is a presheaf on $\wt\Cor_K^X$, the result follows.
\end{proof}

\begin{lemma}[Zariski excision on $\A^1_K$]\label{lemma:Zar-exc-A1K}
Let $x\in\A^1_K$ be a closed point, and let $(\A^1_k,U)\in\wt\Cor^X_K$ be such that $x\in U$. Denote by $i\colon U\hookrightarrow\A^1_K$ the inclusion. Then $i$ induces an isomorphism
\[
i^*\colon\dfrac{\scrF^X(\A^1_K\setminus x)}{\scrF^X(\A^1_K)}\xrightarrow{\cong}\dfrac{\scrF^X(U\setminus x)}{\scrF^X(U)}
\]
 for any homotopy invariant presheaf $\scrF\in\wt\PSh(k)$. 
\end{lemma}

\begin{proof}
This follows similarly as in \Cref{lemma:inj-A1K} above. 
\end{proof}

\section{Injectivity for local schemes}\label{section:injectivity}
The goal of this section is to prove the following theorem.

\begin{theorem}\label{thm:inj-loc-sch}
Let $X$ be a smooth $k$-scheme and $x\in X$ a closed point. Let $U\defeq\Spec(\calO_{X,x})$ and write $\can\colon U\rightarrow X$ for the canonical inclusion. Suppose that $i\colon Z\to X$ is a closed subscheme of codimension $\ge1$ in $X$ satisfying $x\in Z$. Let $j\colon X\setminus Z\hookrightarrow X$ denote the open complement. Then there exists a finite $\rmMW$-correspondence $\Phi\in\wt\Cor_k(U,X\setminus Z)$ such that the diagram
\[\begin{tikzcd}
& X\setminus Z\ar[d,"j"]\\
U\ar[ur,"\Phi"]\ar[r,"\can"] & X
\end{tikzcd}\]
commutes in $\hwtCor_k$.
\end{theorem}
For homotopy invariant presheaves on $\wt\Cor_k$ we immediately obtain:

\begin{corollary}\label{cor:vanishing-section}
Suppose that $\scrF\in\wt\PSh(k)$ is a homotopy invariant presheaf with $\rmMW$-transfers. If $s\in \scrF(X)$ is a section such that $s|_{X\setminus Z}=0$, then $s|_U=0$.
\end{corollary}

Let $X^\circ\subseteq X$ be a Zariski open neighborhood of the point $x$, and let $Z^\circ\defeq Z\cap X^\circ$. As noted in \cite[§8]{hty-inv}, it is enough to solve the problem for the triple $U$, $X^\circ$ and $X^\circ\setminus Z^\circ$. In particular, we may assume that $X$ is irreducible and that the canonical sheaf $\omega_{X/k}$ is trivial. In fact, we will shrink $X$ so that we are in the situation of a relative curve over a quasi-projective scheme. The advantage of this approach is that it turns problems regarding subschemes of high codimension into problems regarding divisors, which is a much more flexible setting. For the shrinking process we refer to the following theorem,  which is originally due to M. Artin.

\begin{theorem}[\protect{\cite[Proposition 1]{PSV}}]
Let $X$, $Z$ and $x\in Z$ be as in \Cref{thm:inj-loc-sch}. Then there is a Zariski open neighborhood $X^\circ\subseteq X$ of the point $x$, an open immersion $X^\circ\hookrightarrow \ol X^\circ$, a Zariski open subscheme $B$ of $\PP^{\dim X-1}$ and a commutative diagram
\[\begin{tikzcd}
  X^\circ\ar[r,hook]\ar[dr,swap,"p"] & \ol X^\circ\ar[d,"\ol p"] & X_\infty^\circ\ar[l,hook']\ar[dl,"p_\infty"]\\
& B &
\end{tikzcd}\]
satisfying the following properties:
\begin{enumerate}
\item[$(1)$] $\ol p$ is a smooth projective morphism, whose fibers are irreducible projective curves.
\item[$(2)$] $X^\circ_\infty=\ol X^\circ\setminus X^\circ$, and $p_\infty\colon X^\circ_\infty\to B$ is finite étale.
\item[$(3)$] The morphism $p|_{Z\cap X^\circ}\colon Z\cap X^\circ\to B$ is finite 
(where the intersection is taken in $\ol X^\circ$). 
\end{enumerate}
The morphism $p\colon X^\circ\to B$ is called an \em{almost elementary fibration}.
\end{theorem}

Following \cite[§8]{hty-inv}, we may shrink $X$ such that there exists an almost elementary fibration $p\colon X\to B$ and such that $\omega_{X/k}$ and $\omega_{B/k}$ are trivial, i.e., $\omega_{X/k}\cong\calO_X$ and $\omega_{B/k}\cong\calO_B$.
Let $\scrX\defeq X\times_BU$ and $\scrZ\defeq Z\times_BU$. Let also $p_X\colon\scrX\to X$ and $p_U\colon\scrX\to U$ be the projections onto $X$ and $U$, respectively, and let $d_X$ denote the dimension of $X$. Finally, let $\Delta$ denote the morphism $\Delta\defeq(\can,\id)\colon U\to \scrX$.

\begin{lemma}[\protect{\cite[Lemma 8.1]{hty-inv}}]\label{lemma:nice-map}
There exists a finite surjective morphism 
\[
H_\theta=(h_\theta,p_U)\colon \scrX\to \A^1\times U
\]
over $U$, such that if we let $\scrD_1\defeq H_\theta^{-1}(1\times U)$ and $\scrD_0\defeq H_\theta^{-1}(0\times U)$ denote the scheme-theoretic preimages, then the following hold:
\begin{enumerate}
\item[$(1)$] $\scrD_1\subseteq \scrX\setminus\scrZ$.
\item[$(2)$] $\scrD_0=\Delta(U)\amalg\scrD_0'$ with $\scrD_0'\subseteq\scrX\setminus\scrZ$.
\end{enumerate}
\end{lemma} 

We will use \Cref{lemma:nice-map} to produce the desired $\rmMW$-correspondence $\Phi$. The aim is to define $\Phi$ as the image $(H_\theta\times1)_*(p_X)$ of the projection $p_X\in\wt\CH^{d_X}_{\Gamma_{p_X}}(\scrX\times X,\omega_X)$ under the pushforward map
\[
(H_\theta\times1)_*\colon\wt\CH^{d_X}_{\Gamma_{p_X}}(\scrX\times X,\omega_{H_\theta\times1}\otimes\omega_X)\to\wt\CH^{d_X}_{(H_\theta\times1)(\Gamma_{p_X})}(\A^1\times U\times X,\omega_X).
\]
To this end, we need a trivialization of $\omega_{H_\theta\times1}=\omega_{\scrX\times X/k}\otimes(H_\theta\times1)^*\omega_{\A^1\times U\times X/k}^\vee$. Now, as $U$ is local we have $\omega_{U/k}\cong\calO_U$. Keeping in mind the discussion preceding \Cref{lemma:nice-map}, it follows that the relative bundle $\omega_{H_\theta\times1}$ is also trivial.
Thus we may choose an isomorphism $\chi\colon\calO_X\cong\omega_{H_\theta\times1}$.

\begin{definition}
Let $p_X\in\wt\Cor_k(\scrX,X)$ denote the projection. Using the trivialization $\chi$ above, we let $\scrH_\theta^\chi\in\wt\Cor_k(\A^1\times U,X)$ denote the image of $p_X\in\wt\Cor_k(\scrX,X)$ under the composition
\begin{align*}
\wt\CH^{d_X}_{\Gamma_{p_X}}(\scrX\times X,\omega_X)&\xrightarrow{\cong}\wt\CH^{d_X}_{\Gamma_{p_X}}(\scrX\times X,\omega_{H_\theta\times1}\otimes\omega_X)\\
&\xrightarrow{(H_\theta\times1)_*}\wt\CH_{(H_\theta\times1)(\Gamma_{p_X})}^{d_X}(\A^1\times U\times X,\omega_X).
\end{align*}
\end{definition}

\begin{lemma}\label{lemma:support}
The morphism $H_\theta\times1$ maps $\Gamma_{p_X}\cong\scrX$ isomorphically onto its image. Let $\scrH^\chi_0\defeq\scrH^\chi_\theta\circ i_0$ and $\scrH^\chi_1\defeq\scrH^\chi_\theta\circ i_1$. Identifying $\scrX$ with its image in $\A^1\times U\times X$, we then have $\supp\scrH^\chi_\theta=\scrX$, $\supp\scrH^\chi_0=\scrD_0$, and $\supp\scrH^\chi_1=\scrD_1$.
\end{lemma}

\begin{proof}
If $y=((x,u),x)$, $y'=((x',u'),x')\in\Gamma_{p_X}$ is such that
\[
(H_\theta\times1)(y)=(h_\theta(x,u),u,x)=(H_\theta\times1)(y')=(h_\theta(x',u'),u',x'),
\]
it follows that $x=x'$ and $u=u'$, hence $y=y'$. Thus we can consider $\scrX$ as a subscheme of $\A^1\times U\times X$ by $(x,u)\mapsto(h_\theta(x,u),u,x)$. Now, the $\rmMW$-correspondence $p_X$ is supported on $\Gamma_{p_X}$, hence $\supp\scrH^\chi_\theta=(H_\theta\times1)(\Gamma_{p_X})\cong\scrX$. We turn to the restrictions $\scrH^\chi_0$ and $\scrH^\chi_1$ of the homotopy $\scrH^\chi_\theta$. By \cite[Example 4.17]{Calmes-Fasel} we have $\scrH^\chi_\theta\circ i_\epsilon=(i_\epsilon\times1)^*(\scrH^\chi_\theta)$, where $\epsilon=0,1$. It follows that $\supp\scrH^\chi_\epsilon=(i_\epsilon\times1)^{-1}((H_\theta\times1)(\Gamma_{p_X})),$ and this closed subset is determined by those points $(x,u)\in\scrX$ satisfying $h_\theta(x,u)=\epsilon$. In other words, $\supp\scrH^\chi_\epsilon=\scrD_\epsilon$.
\end{proof}

\begin{lemma}\label{lemma:finding-Phi}
There is an invertible regular function $\lambda$ on $U$ such that \[\scrH_0^\chi=\can\circ\ip{\lambda}+j\circ \Phi_0'\] and \[\scrH_1^\chi=j\circ\Phi_1,\] where $\Phi_0',\Phi_1\in\wt\Cor_k(U,X\setminus Z)$ and $\ip\lambda\in\K^{\rmMW}_0(U)$.
\end{lemma}

\begin{proof}
By Lemmas \ref{lemma:nice-map} and \ref{lemma:support} we have $\supp\scrH^\chi_0=\Delta(U)\amalg\scrD_0'$, where $\scrD_0'\subseteq\scrX\setminus\scrZ$. By \Cref{lemma:conn-comp-sum} we may therefore write $\scrH^\chi_0=\alpha+\beta$ where $\alpha\in\wt\Cor_k(U,X)$ is supported on $\Delta(U)$ and $\beta\in\wt\Cor_k(U,X)$ is supported on $\scrD_0'$. Since $\supp\beta=\scrD_0'\subseteq\scrX\setminus\scrZ$, \Cref{cor:supp-open} ensures that there exists a unique finite $\rmMW$-correspondence  $\Phi_0'\in\wt\Cor_k(U,X\setminus Z)$ such that $j\circ\Phi_0'=\beta$. Hence $\scrH^\chi_0$ is of the form $\scrH^\chi_0=\alpha+j\circ\Phi_0'$ for $\Phi_0'\in\wt\Cor_k(U,X\setminus Z)$. The same reasoning shows that, since $\supp\scrH^\chi_1= \scrD_1\subseteq\scrX\setminus\scrZ$, there is a unique $\rmMW$-correspondence $\Phi_1\in\wt\Cor_k(U,X\setminus Z)$ such that $\scrH^\chi_1=j\circ\Phi_1$. 

It therefore only remains to understand the finite $\rmMW$-correspondence $\alpha\in\wt\CH_{\Delta(U)}^{d_X}(U\times X,\omega_{X})$. Recall that, by definition, 
\[
\scrH^\chi_0=(i_0\times1)^*(H_\theta\times1)_*(\Gamma_{p_X})_*(\ip1).
\]
Let $i_{\Delta(U)}$ and $i_{\scrD_0}$ denote the respective inclusions $i_{\Delta(U)}\colon\Delta(U)\subseteq \scrX$ and $i_{\scrD_0}\colon\scrD_0\subseteq\scrX$. The base change formula \cite[Proposition 3.2]{Calmes-Fasel} applied to the pullback square
\[\begin{tikzcd}
(\Delta(U)\amalg\scrD_0')\times X\ar{r}{i_{\scrD_0}\times1}\ar{d}[swap]{H_\theta|_{\scrD_0}\times1} & \scrX\times X\ar{d}{H_\theta\times1}\\
U\times X\ar{r}{i_0\times1} & \A^1\times U\times X
\end{tikzcd}\]
reveals that $\alpha=(H_\theta|_{\Delta(U)}\times1)_*(i_{\Delta(U)}\times1)^*(\Gamma_{p_X})_*(\ip1)$. Using that $\Delta\colon U\to\scrX$ is an isomorphism onto its image and that $H_\theta|_{\Delta(U)}\colon\Delta(U)\to U$ is an isomorphism, we may write $\alpha=(\Delta\times1)^*(\Gamma_{p_X})_*(\ip1)$. Next, consider the pullback diagram
\[\begin{tikzcd}
U\ar{rr}{\Delta}\ar{d}[swap]{\Gamma_\can} & & \scrX\ar{d}{\Gamma_{p_X}}\\
U\times X\ar{rr}{\Delta\times1} & &\scrX\times X.
\end{tikzcd}\]
Using base change once more, we obtain $\alpha=(\Gamma_\can)_*\Delta^*(\ip1)$. Comparing this expression with the definition $\wt\gamma_\can\defeq(\Gamma_\can)_*(\ip1)$ of $\wt\gamma_\can$, we see that two possibly different trivializations of the line bundle $\omega_U$ are involved. Letting $\lambda\in k[U]^\times$ be the fraction of these two trivializations, it follows that $\alpha=\wt\gamma_\can\circ\ip{\lambda}$.
\end{proof}

\begin{proof}[Proof of \Cref{thm:inj-loc-sch}]
In the notation of \Cref{lemma:finding-Phi}, define $\scrH_\theta\defeq\scrH_\theta^\chi\circ\ip{\lambda^{-1}}$ and $\Phi\defeq(\Phi_1-\Phi_0')\circ\ip{\lambda^{-1}}$. By \Cref{lemma:finding-Phi}, $\scrH_\theta$ provides a homotopy $\can\sim_{\A^1}j\circ\Phi$.
\end{proof}

\section{Nisnevich excision}\label{section:Nis-excision}
The setting of this section is as follows. Suppose that $X,X'\in\Sm_k$ are smooth affine $k$-schemes such that there is an elementary distinguished Nisnevich square
\begin{equation}
\begin{tikzcd}
V'\ar{r}\ar{d} & X'\ar[d,"\Pi"]\\
V\ar{r} & X.
\end{tikzcd}\label{eq:Nis-square}
\end{equation}
Define the closed subschemes $S\defeq (X\setminus V)_\red\subseteq X$ and $S'\defeq(X'\setminus V')_\red\subseteq X'$. Let $x\in S$ and $x'\in S'$ be two points satisfying $\Pi(x')=x$. Moreover, we set $U\defeq\Spec(\calO_{X,x})$ and $U'\defeq\Spec(\calO_{X',x'})$. Let $\can\colon U\to X$ and $\can'\colon U'\to X'$ be the canonical inclusions and let $\pi\defeq\Pi|_{U'}\colon U'\to U$. We can summarize the situation with the following diagram:
\begin{equation}
\begin{tikzcd}
V'\ar{r}\ar{d} & X'\ar[d,"\Pi"] & U'\ar[l,swap,"\can'"]\ar[d,"\pi"]\\
V\ar{r} & X & U.\ar[l,"\can"]
\end{tikzcd}\label{eq:Nis-square2}
\end{equation}
The main result of this section is the following excision theorem for Nisnevich squares.

\begin{theorem}[Nisnevich excision]\label{thm:Nis-excision}
Let $\scrF$ be a homotopy invariant presheaf on $\wt\Cor_k$. Given any elementary distinguished Nisnevich square as \eqref{eq:Nis-square}, suppose in addition that the closed subscheme $S$ is smooth over $k$. Then the induced morphism
\[
\pi^*\colon \dfrac{\scrF(U\setminus S)}{\scrF(U)}\to \dfrac{\scrF(U'\setminus S')}{\scrF(U')}
\]
is an isomorphism.
\end{theorem}

The proof of \Cref{thm:Nis-excision} relies on the two following results, establishing respectively injectivity and surjectivity of $\pi^*$:

\begin{theorem}[Injectivity of Nisnevich excision]\label{thm:Nis-inj}
With the notations in \eqref{eq:Nis-square2}, there exist finite $\rmMW$-correspondences $\Phi\in\wt\Cor_k^\pair((U,U\setminus S),(X',X'\setminus S'))$ and $\Theta\in\wt\Cor_k^\pair((U,U\setminus S),(X\setminus S,X\setminus S))$ such that 
\[
\ol\Pi\circ\ol\Phi-\ol j_X\circ\ol\Theta=\ol\can
\]
in $\hwtCor_k^\pair((U,U\setminus S),(X,X\setminus S))$. Here $j_X\colon(X\setminus S,X\setminus S)\hookrightarrow(X,X\setminus S)$ is the inclusion.
\end{theorem}

\begin{theorem}[Surjectivity of Nisnevich excision]\label{thm:Nis-surj}
With the notations in \eqref{eq:Nis-square2}, assume in addition that $S$ is smooth over $k$. Then there exist finite $\rmMW$-correspondences $\Psi\in\wt\Cor^\pair_k((U,U\setminus S),(X',X'\setminus S'))$ and $\Xi\in\wt\Cor_k^\pair((U',U'\setminus S'),(X'\setminus S',X'\setminus S'))$ such that 
\[
\ol\Psi\circ\ol\pi-\ol j_{X'}\circ\ol{\Xi}=\ol{\can'}
\]
in $\hwtCor_k^\pair((U',U'\setminus S'),(X',X'\setminus S'))$. Here $j_{X'}\colon (X'\setminus S',X'\setminus S')\hookrightarrow (X',X'\setminus S')$ is the inclusion.
\end{theorem}

Assuming Theorems \ref{thm:Nis-inj} and \ref{thm:Nis-surj}, \Cref{thm:Nis-excision} now follows:

\begin{proof}[Proof of \Cref{thm:Nis-excision}]\label{pf:Nis-exc}
  Let $\scrF$ be a homotopy invariant presheaf with $\rmMW$-transfers. First, note that \Cref{thm:inj-loc-sch} implies that the restriction maps $\scrF(U)\to \scrF(U\setminus S)$ and $\scrF(U')\to\scrF(U'\setminus S')$ are injective. Indeed, suppose that $s_x\in\scrF(U)$ maps to $0$ in $\scrF(U\setminus S)$. We may assume that $s_x$ is represented by a section $s\in\scrF(W)$ for some Zariski open neighborhood $W$ of $x$, such that $s|_{W\setminus S}=0$. But then $s_x=0$ by \Cref{cor:vanishing-section}. Hence $\scrF(U)\to \scrF(U\setminus S)$ is injective. It follows similarly that $\scrF(U')\to \scrF(U'\setminus S')$ is injective.

  Now, as the MW-correspondence $\Theta$ of \Cref{thm:Nis-inj} maps to $(X\setminus S,X\setminus S)$, $j_X\circ\Theta$ induces the trivial map
\[
(j_X\circ\Theta)^*=0\colon \dfrac{\scrF(X\setminus S)}{\scrF(X)}\to \dfrac{\scrF(U\setminus S)}{\scrF(U)}.
\]
Hence $\Phi^*\circ\Pi^*=\can^*$. Similarly, $\Xi^*=0$ and hence $\pi^*\circ\Psi^*=(\can')^*$. We use this to show that $\pi^*$ is an isomorphism.

To show that $\pi^*$ is injective, let us assume that $s_x\in\scrF(U\setminus S)/\scrF(U)$ is a germ such that $\pi^*(s_x)=0$. As 
\[
\dfrac{\scrF(U\setminus S)}{\scrF(U)}=\varinjlim_{W\ni x}\dfrac{\scrF(W\setminus S)}{\scrF(W)},
\]
we may assume that $s_x$ is represented by a section $s\in\scrF(W\setminus S)/\scrF(W)$ for some affine $k$-smooth Zariski open neighborhood $W$ of $x$. 
Thus $s$ is a section satisfying $\can^*(s)=s_x$ and $\pi^*(s_x)=0$. Now, since $\pi^*(s_x)=0$ in $\scrF(U'\setminus S')/\scrF(U')$, there is some affine $k$-smooth Zariski open neighborhood $W'$ of $x'$ in $X'\times_X W$ such that $\Pi^*(s)|_{W'}=0$. Replacing $X$ by $W$ and $X'$ by $W'$, we may then apply \Cref{thm:Nis-inj} to obtain a finite $\rmMW$-correspondence $\Phi\in\wt\Cor_k(U,X')$ such that $\Phi^*\circ\Pi^*=\can^*$. But then $s_x=\can^*(s)=\Phi^*(\Pi^*(s))=0$. Hence $\pi^*$ is injective.

To show surjectivity, let $s'_{x'}\in\scrF(U'\setminus S')/\scrF(U')$. Similarly as above, we may assume that $s'_{x'}$ is represented by a section $s'\in\scrF(X'\setminus S')/\scrF(X')$, i.e., $(\can')^*(s')=s'_{x'}$. By \Cref{thm:Nis-surj}, there is a finite MW-correspondence $\Psi\in\wt\Cor_k(U,X')$ such that $\pi^*\circ\Psi^*=(\can')^*$. We then have $s_{x'}'=(\can')^*(s')=\pi^*(\Psi^*(s'))$, and thus $\pi^*$ is surjective.
\end{proof}

We proceed to prove Theorems \ref{thm:Nis-inj} and \ref{thm:Nis-surj}.

\section{Injectivity of Nisnevich excision}\label{section:Nis-inj}
In this section we aim to prove \Cref{thm:Nis-inj}. As preparation, we need to perform a shrinking process similar to that in \Cref{section:injectivity}. By \cite[Lemma 9.4]{hty-inv}, there is a Zariski open subscheme $X^\circ\subseteq X$ along with an almost elementary fibration $q\colon X^\circ\to B$ such that $\omega_{B/k}\cong \calO_B$ and $\omega_{X^\circ/k}\cong\calO_{X^\circ}$. By \cite[§9]{hty-inv} we may replace $X$ by $X^\circ$ and $X'$ by $\Pi^{-1}(X^\circ)$. We regard $X'$ as a $B$-scheme via the map $q\circ\Pi$. 

Let $\Delta$ denote the morphism $\Delta\defeq(\id,\can)\colon U\to U\times_B X$, and let $p_X$ and $p_{\A^1\times U}$ denote the projections from $\A^1\times U\times_B X$ onto $X$ respectively $\A^1\times U$.

\begin{prop}[\protect{\cite[Proposition 9.9]{hty-inv}}]\label{prop:fu}
Let $\theta$ be the coordinate of $\A^1$. There exists a function $h_\theta\in k[\A^1\times U\times_B X]$ such that the following properties hold for the functions $h_\theta$, $h_0\defeq h_\theta|_{0\times U\times_B X}$ and $h_1\defeq h_\theta|_{1\times U\times_B X}$:
\begin{enumerate}
\item[$(a)$] The morphism $H_\theta\defeq(p_{\A^1\times U},h_\theta)\colon\A^1\times U\times_B X\to \A^1\times U\times\A^1$ is finite and surjective. Letting $Z_\theta\defeq h_\theta^{-1}(0)\subseteq \A^1\times U\times_B X$, it follows that $Z_\theta$ is finite, surjective and flat over $\A^1\times U$.

\item[$(b)$] Let $Z_0\defeq h_0^{-1}(0)\subseteq U\times_B X$. Then there is the equality of schemes $Z_0=\Delta(U)\amalg G$, where $G\subseteq U\times_B(X\setminus S)$.
\item[$(c)$] The closed subscheme $\calV((\id_U\times\Pi)^*(h_1))\subseteq U\times_B X'$ is a disjoint union of two closed subschemes $Z_1'\amalg Z_2'$. Moreover, the map $(\id_U\times\Pi)|_{Z_1'}$ identifies $Z_1'$ with $Z_1\defeq h^{-1}_1(0)$.
\[\begin{tikzcd}
U\times_B X'\ar[r,"1\times\Pi"] & U\times_B X\ar[r,"h_1"] & \A^1\\
Z_1'\ar[u,hook]\ar[r,"\cong"] & Z_1=\calV(h_1)\ar[u,hook]
\end{tikzcd}\]
\item[$(d)$] We have $Z_\theta\cap(\A^1\times(U\setminus x)\times_B X)\subseteq\A^1\times(U\setminus x)\times_B (X\setminus x)$.
\end{enumerate}
\end{prop}

\begin{corollary}[\protect{\cite[Remark 9.10]{hty-inv}}]\label{cor:subsets}
We have the following inclusions:
\begin{enumerate}
\item[$(1)$] $Z_\theta\cap(\A^1\times(U\setminus S)\times_B X)\subseteq \A^1\times(U\setminus S)\times_B X\setminus S$.
\item[$(2)$] $Z_0\cap((U\setminus S)\times_B X)\subseteq(U\setminus S)\times_B(X\setminus S)$.
\item[$(3)$] $Z_1\cap((U\setminus S)\times_B X)\subseteq(U\setminus S)\times_B(X\setminus S)$.
\item[$(4)$] $Z_1'\cap((U\setminus S)\times_B X')\subseteq(U\setminus S)\times_B(X'\setminus S')$.
\end{enumerate}
\end{corollary}

\begin{definition}\label{def:homotopy}
Choose a trivialization $\chi$ of $\omega_{H_\theta\times1}$. We define $\scrH_\theta^\chi\in\wt\Cor_k(\A^1\times U,X)$ as the image of the projection $p_X\in\wt\CH^{d_X}_{\Gamma_{p_X}}(\A^1\times U\times_B X\times X,\omega_X)$ under the composition
\begin{align*}
&\wt\CH^{d_X}_{\Gamma_{p_X}}(\A^1\times U\times_BX\times X,\omega_X)\\
&\xrightarrow{\cong} \wt\CH_{\Gamma_{p_X}}^{d_X}(\A^1\times U\times_B X\times X,\omega_{H_\theta\times1}\otimes\omega_X)\\
&\xrightarrow{(H_\theta\times1)_*} \wt\CH^{d_X}_{(H_\theta\times1)(\Gamma_{p_X})}(\A^1\times U\times \A^1\times X,\omega_X)\\
&\xrightarrow{(1\times i_0\times1)^*} \wt\CH^{d_X}_{T}(\A^1\times U\times X,\omega_X),
\end{align*}
 where $d_X\defeq\dim X$, $T\defeq(1\times i_0\times 1)^{-1}((H_\theta\times1)(\Gamma_{p_X}))$, and where the first isomorphism is induced by $\chi$.
\end{definition}

\begin{lemma}\label{lemma:supp}
The finite $\rmMW$-correspondence $\scrH^\chi_\theta$ is supported on $Z_\theta$. Moreover, for $\epsilon=0,1$ we have $\supp\scrH^\chi_\epsilon=Z_\epsilon$ (where $\scrH^\chi_\epsilon\defeq\scrH^\chi_\theta\circ i_\epsilon$).
\end{lemma}

\begin{proof}
Let $T$ denote the support of $\scrH^\chi_\theta$. As indicated in \Cref{def:homotopy} we have $T=(1\times i_0\times 1)^{-1}((H_\theta\times1)(\Gamma_{p_X}))$. By the same argument as in \Cref{lemma:support}, $H_\theta\times1$ injects $\Gamma_{p_X}$ onto its image, hence $(H_\theta\times1)(\Gamma_{p_X})\cong\A^1\times U\times_B X$. Thus $T$ consists of those points $(t,u,x)\in\A^1\times U\times_B X$ such that $h_\theta(t,u,x)=0$, i.e., $T=Z_\theta$.

Turning to the support of $\scrH^\chi_\epsilon$, note that $\scrH^\chi_\epsilon$ is the image of $p_X$ under the composition
\[
\wt\CH^{d_X}_{\Gamma_{p_X}}(\A^1\times U\times_B X\times X,\omega_X)\xrightarrow{(i_\epsilon\times1)^*\circ(1\times i_0\times1)^*\circ(H_\theta\times1)_*}\wt\CH^{d_X}_{\supp\scrH^\chi_\epsilon}(\epsilon\times U\times X,\omega_X).
\]
By the same reasoning as above, pulling back along $i_\epsilon\times1$ amounts to substituting $\theta=\epsilon$ in $h_\theta$, which yields the desired result.
\end{proof}

\begin{lemma}\label{lemma:hty-0-1}
There are finite $\rmMW$-correspondences $\Theta\in\wt\Cor_k(U,X\setminus S)$ and $\Phi\in\wt\Cor_k(U,X')$ along with an invertible regular function $\lambda$ on $U$ such that $\scrH^\chi_0=\can\circ\ip{\lambda}+j_X\circ\Theta$ and $\scrH^\chi_1=\Pi\circ\Phi$.
\end{lemma}

\begin{proof}
  By \Cref{prop:fu} $(b)$, we can write $\scrH^\chi_0=\alpha+\Theta'$, where $\Theta'\in\wt\Cor_k(U,X)$ is supported on $G$ and $\alpha\in\wt\Cor_k(U,X)$ is supported on $\Delta(U)$. Using \Cref{prop:fu} $(b)$, \Cref{lemma:iso} ensures that there is a unique finite $\rmMW$-correspondence $\Theta\in\wt\Cor_k(U,X\setminus S)$ such that $\Theta'=j_X\circ\Theta$. We proceed similarly for $\scrH^\chi_1$: by \Cref{prop:fu} $(c)$, the pullback 
\[
(1\times\Pi)^*(\scrH^\chi_1)\in\wt\CH^{d_X}_{(1\times\Pi)^{-1}(Z_1)}(U\times X',\omega_{X'})
\]
is supported on $Z_1'\amalg Z_2'$, and $(1\times\Pi)|_{Z_1'}$ is an isomorphism from $Z_1'$ onto $Z_1$. It follows that we have an isomorphism
\[
(1\times\Pi)_*\colon\wt\CH_{Z_1'}^{d_X}(U\times X',\omega_{X'})\xrightarrow{\cong}\wt\CH_{Z_1}^{d_X}(U\times X,\omega_X).
\]
Hence $\Phi\defeq(1\times\Pi)_*^{-1}(\scrH^\chi_1)=(1\times\Pi)^*(\scrH^\chi_1)\in\wt\Cor_k(U,X')$ satisfies $\Pi\circ\Phi=\scrH^\chi_1$.

It remains to show that $\alpha=\can\circ\ip{\lambda}$, the proof of which being similar as in the proof of \Cref{lemma:finding-Phi}. As 
\[
(1\times i_0\times1)\circ(i_0\times1)=(i_0\times1\times i_0\times1)\colon U\times X\to\A^1\times U\times\A^1\times X,
\]
we can write $\scrH^\chi_0=(i_0\times1\times i_0\times1)^*(H_\theta\times1)_*(\Gamma_{p_X})_*(\ip1)$. Using the base change formula twice as in \Cref{lemma:finding-Phi}, we find that 
\[
\alpha=(H_\theta|_{\Delta(U)}\times1)_*(i_{\Delta(U)}\times1)^*(\Gamma_{p_X})_*(\ip1)=(\Gamma_\can)_*(\ip1)\circ\ip{\lambda}=\wt\gamma_\can\circ\ip{\lambda},
\]
where $\lambda\in k[U]^\times$ is the fraction of two trivializations of $\omega_U$, and $i_{\Delta(U)}\colon\Delta(U)\hookrightarrow U\times_B X$ is the inclusion.
\end{proof}

\begin{lemma}\label{lemma:bunch-of-pairs}
Let $j_U\colon U\setminus S \hookrightarrow U$, $j_X\colon X\setminus S\hookrightarrow X$ and $j_{X'}\colon X'\setminus S'\hookrightarrow X'$ denote the inclusions, and set: 
\begin{align*}
\scrH^{\chi,S}_\theta&\defeq(1\times j_U\times j_X)^*(\scrH^\chi_\theta)\in\wt\Cor_k(\A^1\times(U\setminus S),X\setminus S).\\
\Phi^S&\defeq(j_U\times j_{X'})^*(\Phi)\in\wt\Cor_k(U\setminus S,X'\setminus S').\\
\Theta^S&\defeq(j_U\times1)^*(\Theta)\in\wt\Cor_k(U\setminus S,X\setminus S).
\end{align*}
Then we have:
\begin{align*}
(\scrH^\chi_\theta,\scrH_\theta^{\chi,S})&\in\wt\Cor_k^\pair(\A^1\times(U,U\setminus S),(X,X\setminus S)).\\
(\Phi,\Phi^S)&\in\wt\Cor_k^\pair((U,U\setminus S),(X',X'\setminus S')).\\
(\Theta,\Theta^S)&\in\wt\Cor_k^\pair((U,U\setminus S),(X\setminus S,X\setminus S)).
\end{align*}
\end{lemma}

\begin{proof}
In light of \Cref{cor:subsets}, this follows from \Cref{lemma:making-pairs}.
\end{proof}

\begin{proof}[Proof of \Cref{thm:Nis-inj}]
Replacing $\scrH_\theta^\chi$, $\Theta$ and $\Phi$ by the respective precompositions with $\ip{\lambda^{-1}}$, it follows from Lemmas \ref{lemma:hty-0-1} and \ref{lemma:bunch-of-pairs} that we have the identity
\[
\ol\Pi\circ\ol\Phi-\ol j_X\circ\ol{\Theta}=\ol\can
\]
in $\hwtCor_k^\pair$.
\end{proof}

\section{Surjectivity of Nisnevich excision}\label{section:Nis-surj}
We proceed to prove \Cref{thm:Nis-surj}. In this section, the closed subscheme $S\subseteq X$ is assumed to be smooth over $k$. Performing a similar shrinking process as in \Cref{section:Nis-inj}, we may assume that there is an almost elementary fibration $q\colon X\to B$ such that $\omega_{B/k}\cong\calO_B$ and $\omega_{X/k}\cong\calO_X$. Since $\Pi$ is étale, it follows that $\omega_{X'/k}\cong\calO_{X'}$.

Let $\Delta'\defeq(\id,\can')\colon U'\to U'\times_B X'$, and let $p_{X'}$ and $p_{\A^1\times U'}$ denote the projections from $\A^1\times U'\times_B X'$ to $X'$ respectively $\A^1\times U'$. First we recall the following fact from \cite{hty-inv}:

\begin{prop}[\protect{\cite[Proposition 11.6]{hty-inv}}]\label{prop:fu2}
Let $\A^1$ have coordinate $\theta$. There exist functions $F\in k[U\times X']$ and $h_\theta'\in k[\A^1\times U'\times_B X']$ such that the following properties hold for the functions $F$, $h_\theta'$, $h'_0\defeq h_\theta'|_{0\times U'\times_B X'}$, and $h_1'\defeq h_\theta'|_{1\times U'\times_B X'}$:
\begin{itemize}
\item[$(a)$] The morphism $H_\theta'\defeq(p_{\A^1\times U'},h_\theta')\colon\A^1\times U'\times_B X'\to \A^1\times U'\times\A^1$ is finite and surjective. Letting $Z'_\theta\defeq (h_\theta')^{-1}(0)\subseteq\A^1\times U'\times_B X'$, it follows that $Z_\theta'$ is finite, surjective and flat over $\A^1\times U'$.
\item[$(b)$] Let $Z_0'\defeq (h_0')^{-1}(0)$. Then there is the equality of schemes $Z_0'=\Delta'(U')\amalg G'$, where $G'\subseteq U'\times_B(X'\setminus S')$.
\item[$(c)$] $h_1'=(\pi\times\id_{X'})^*(F)$. We write $Z_1'\defeq(h_1')^{-1}(0)$.
\item[$(d)$] $Z_\theta'\cap(\A^1\times(U'\setminus S')\times_B X')\subseteq\A^1\times(U'\setminus S')\times_B(X'\setminus S')$.
\item[$(e)$] The morphism $(\pr_U,F)\colon U\times X'\to U\times\A^1$ is finite and surjective. Letting $Z_1\defeq F^{-1}(0)$, it follows that $Z_1$ is finite and surjective over $U$.
\item[$(f)$] $Z_1\cap((U\setminus S)\times X')\subseteq(U\setminus S)\times(X'\setminus S')$.
\end{itemize}
\end{prop}

\begin{corollary}[\protect{\cite[Remark 11.7]{hty-inv}}]\label{cor:subsets2}
We have the following inclusions:
\begin{enumerate}
\item[$(1)$] $Z_\theta'\cap(\A^1\times(U'\setminus S')\times_B X')\subseteq\A^1\times(U'\setminus S')\times_B(X'\setminus S')$.
\item[$(2)$] $Z_0'\cap((U'\setminus S')\times_B X')\subseteq(U'\setminus S')\times_B(X'\setminus S')$.
\item[$(3)$] $Z_1'\cap((U'\setminus S')\times_B X')\subseteq(U'\setminus S')\times_B(X'\setminus S')$.
\item[$(4)$] $Z_1\cap((U\setminus S)\times X')\subseteq(U\setminus S)\times(X'\setminus S')$.
\end{enumerate}
\end{corollary}

\begin{definition}\label{def:homotopy2}
Choose a trivialization $\chi$ of $\omega_{H_\theta'\times1}$. We define $\scrH^{\chi}_\theta\in\wt\Cor_k(\A^1\times U',X')$ as the image of the projection $p_{X'}\in\wt\CH^{d_X}_{\Gamma_{p_{X'}}}(\A^1\times U'\times_B X'\times X',\omega_{X'})$ under the composition
\begin{align*}
&\wt\CH^{d_X}_{\Gamma_{p_{X'}}}(\A^1\times U'\times_BX'\times X',\omega_{X'})\\
&\xrightarrow{\cong} \wt\CH_{\Gamma_{p_{X'}}}^{d_X}(\A^1\times U'\times_B X'\times X',\omega_{H_\theta'\times1}\otimes\omega_{X'})\\
&\xrightarrow{(H_\theta'\times1)_*} \wt\CH^{d_X}_{(H_\theta'\times1)(\Gamma_{p_{X'}})}(\A^1\times U'\times \A^1\times X',\omega_{X'})\\
&\xrightarrow{(1\times i_0\times1)^*} \wt\CH^{d_X}_{T'}(\A^1\times U'\times X',\omega_{X'}),
\end{align*}
 where  $T'\defeq(1\times i_0\times 1)^{-1}((H_\theta'\times1)(\Gamma_{p_{X'}}))$, and where the first isomorphism is induced by $\chi'$.
\end{definition}

The same argument as in \Cref{lemma:supp} readily yields:

\begin{lemma}
The finite $\rmMW$-correspondence $\scrH^\chi_\theta$ is supported on $Z_\theta'$. Moreover, for $\epsilon=0,1$ we have $\supp\scrH^\chi_\epsilon=Z_\epsilon'$ (where, as usual, $\scrH^\chi_\epsilon\defeq\scrH^\chi_\theta\circ i_\epsilon$).
\end{lemma}

\begin{lemma}\label{lemma:hty-0-12}
There are finite $\rmMW$-correspondences $\Psi'\in\wt\Cor_k(U,X')$ and $\Xi\in\wt\Cor_k(U',X'\setminus S')$ along with an invertible regular function $\lambda'$ on $U'$ such that $\scrH^\chi_0=\can'\circ\ip{\lambda'}+j_{X'}\circ\Xi$ and $\scrH^{\chi}_1=\Psi'\circ\pi$.
\end{lemma}

\begin{proof}
The claim about $\scrH^{\chi}_0$ follows from an identical argument as in the proof of \Cref{lemma:hty-0-1} by using \Cref{prop:fu2} $(b)$, so let us turn our attention to $\scrH^\chi_1$. By \Cref{prop:fu2} $(c)$, the morphism $\pi\times1$ identifies $Z_1'$ with $Z_1$. By étale excision \cite[Lemma 3.7]{Calmes-Fasel}, $\pi\times1$ induces an isomorphism
\[
(\pi\times1)^*\colon\wt\CH_{Z_1}^{d_X}(U\times X',\omega_{X'})\xrightarrow{\cong}\wt\CH_{Z_1'}^{d_X}(U'\times X',\omega_{X'}).
\]
Arguing similarly to the proof of \Cref{lemma:hty-0-1}, it follows that there exists a unique element $\Psi'\in\wt\CH_{Z_1'}^{d_X}(U'\times X',\omega_{X'})\subseteq\wt\Cor_k(U',X')$ such that $\scrH^\chi_1=\Psi'\circ\pi$.
\end{proof}

Let us check also that the finite $\rmMW$-correspondences constructed above are in fact morphisms of pairs:

\begin{lemma}\label{lemma:bunch-of-pairs2}
Let $j_{U'}\colon U'\setminus S'\hookrightarrow U'$ denote the inclusion, and define:
\begin{align*}
\scrH^{\chi,S'}_\theta &\defeq(1\times j_{U'}\times j_{X'})^*(\scrH^\chi_\theta)\in\wt\Cor_k(\A^1\times (U'\setminus S'),X'\setminus S').\\
\Psi^{S'}&\defeq (j_U\times j_{X'})^*(\Psi)\in\wt\Cor_k(U\setminus S,X'\setminus S').\\
\Xi^{S'}&\defeq(j_{U'}\times1)^*(\Xi)\in \wt\Cor_k(U'\setminus S',X'\setminus S').
\end{align*}
Then
\begin{align*}
(\scrH^\chi_\theta,\scrH^{\chi,S'}_\theta)&\in\wt\Cor_k^\pair(\A^1\times(U',U'\setminus S'),(X',X'\setminus S')).\\
(\Psi,\Psi^{S'})&\in\wt\Cor_k^\pair((U,U\setminus S),(X',X'\setminus S')).\\
(\Xi,\Xi^{S'})&\in\wt\Cor_k^{\pair}((U',U'\setminus S'),(X'\setminus S',X'\setminus S')).
\end{align*}
\end{lemma}

\begin{proof}
By \Cref{cor:subsets2}, the supports of the given $\rmMW$-correspondences satisfy the hypothesis of \Cref{lemma:making-pairs}.
\end{proof}

We are almost in position to prove \Cref{thm:Nis-surj}. However, as opposed to the situation in \Cref{section:Nis-inj} we cannot immediately precompose the homotopy $\scrH_\theta^\chi$ of \Cref{def:homotopy2} with $\ip{(\lambda')^{-1}}$ and obtain a homotopy of the desired form. In order to remedy this, we need the following lemma (see also \cite[Proof of Proposition 6.7]{DruGW}):

\begin{lemma}\label{lemma:trick}
Let $X$, $S$, $U$ and $\can$ be as in \Cref{thm:Nis-excision}, and suppose that $\lambda\in k[U]^\times$ is an invertible regular function satisfying $\lambda|_{U\cap S}=1$. Then
\[\can\circ\ip{\lambda}\sim_{\A^1}\can\in\wt\Cor_k^\pair((U,U\setminus S),(X,X\setminus S)).\]
\end{lemma}

\begin{proof}
The germ $\lambda$ is represented by an invertible section $\mu$ on some smooth affine Zariski open neighborhood $W$ of $x$ in $X$. Moreover, by assumption there is some affine Zariski open neighborhood $W'\subseteq W$ of $x$ such that $\mu|_{S\cap W'}=1$. Replacing $X$ by $W'$, we may assume that $\mu$ is an invertible regular function on the smooth affine $k$-scheme $X$ satisfying $\mu|_S=1$. 

Define an étale covering $\Pi\colon X'\to X$ by letting $X'\defeq\Spec(k[X][t]/(t^2-\mu))$. Consider the closed subscheme $S'\defeq \Spec(k[S][t]/(t-1))$ of $X'$. As $\Pi$ induces an isomorphism $S'\xrightarrow{\cong} S$, it follows that we have an elementary distinguished Nisnevich square
\[\begin{tikzcd}
X'\setminus S'\ar{r}\ar{d} & X'\ar[d,"\Pi"]\\
X\setminus S\ar{r} & X.
\end{tikzcd}\]
Thus \Cref{thm:Nis-inj} provides us with a finite MW-correspondence $\Phi\in\wt\Cor_k(U,X')$ such that $\Pi\circ\Phi\sim_{\A^1}\can$ as correspondences of pairs. But then
\begin{align*}
\can\circ\ip\lambda &= \ip{\mu}\circ\can\\
&\sim_{\A^1}\ip{\mu}\circ\Pi\circ\Phi\\
&= \Pi\circ\ip{\Pi^*(\mu)}\circ\Phi\\
&=\Pi\circ\ip{t^2}\circ\Phi\\
&=\Pi\circ\Phi\sim_{\A^1}\can,
\end{align*}
where we have used the left and right actions of $\K_0^{\rmMW}$ on finite MW-correspondences \cite[Example 4.14]{Calmes-Fasel}, along with the fact that $\ip{a^2}=1$ in $\rmK_0^{\rmMW}$.
\end{proof}

\begin{proof}[Proof of \Cref{thm:Nis-surj}]
Using that $k(x)\cong k(x')$, we can find an invertible regular function $\nu$ on $U$ such that $\pi^*(\nu)(x')=\lambda'(x')^{-1}$. Then put
\[
\scrH_\theta\defeq\scrH_\theta^{\chi}\circ\ip{\pi^*(\nu)}
\]
and
\[
\Psi\defeq\Psi'\circ\ip\nu.
\]
By Lemmas \ref{lemma:hty-0-12} and \ref{lemma:bunch-of-pairs2}, $\scrH_\theta$ provides a homotopy of correspondences of pairs
\[
\can'\circ\ip{\lambda'\cdot\pi^*(\nu)}\sim_{\A^1}\Psi'\circ\pi\circ\ip{\pi^*(\nu)}=\Psi'\circ\ip\nu\circ\pi=\Psi\circ\pi.
\]
We conclude by noting that $\can'\circ\ip{\lambda'\cdot\pi^*(\nu)}\sim_{\A^1}\can'$ by \Cref{lemma:trick}.
\end{proof}

\section{Homotopy invariance}\label{section:hty-inv}

In this section we show, following \cite[Proof of Theorem 2.1]{hty-inv} and \cite{Dru14}, how homotopy invariance of the sheaves $\scrF_\Zar$ and $\scrF_\Nis$ follows from the excision theorems along with injectivity for local schemes. Throughout this section $\scrF$ will denote a homotopy invariant presheaf with $\rmMW$-transfers, and $X\in\Sm_k$ will denote a smooth irreducible $k$-scheme with generic point $\eta\colon \Spec(k(X))\to X$. Write $K\defeq k(X)$ for the function field of $X$.

In this section, the field $k$ is assumed to be of characteristic $0$.

\subsection*{Homotopy invariance of $\scrF_\Zar$}Below we will use Zariski excision along with injectivity for local schemes to show homotopy invariance of the Zariski sheaf $\scrF_\Zar$ associated to $\scrF$. Let $x\in X$ be a closed point of $X$. We may write $\scrF(\Spec(\calO_{X,x}))$ or $\scrF(\calO_{X,x})$ for the stalk $\scrF_x$ of $\scrF$ at $x$ in the Zariski topology.

\begin{lemma}\label{lemma:injective}
The natural map $\eta^*\colon\scrF(\calO_{X,x})\to \scrF(K)$ is injective.
\end{lemma}

\begin{proof}
For $U\defeq\Spec(\calO_{X,x})$ we have $\scrF(U)=\varinjlim_{V\ni x}\scrF(V)$, and $\scrF(K)=\varinjlim_{W\ne\varnothing}\scrF(W)$. Let $s_x\in \scrF(\calO_{X,x})$ be a germ mapping to $0$ in $\scrF(K)$. This means that there is some nonempty open $W\subseteq X$ such that $s|_W=0$. If $x\in W$ then $s_x=0$ in $\scrF(\calO_{X,x})$ and we are done. So suppose that $x\not\in W$, and let $Z$ denote the closed complement of $W$ in $X$. Then $s|_{X\setminus Z}=0$, and thus \Cref{cor:vanishing-section} applies, yielding $s_x=0$ in $\scrF(\calO_{X,x})$.
\end{proof}

\begin{corollary}\label{cor:gen-stalk}
The map $\eta^*\colon \scrF_\Zar(X)\to \scrF_\Zar(K)$ is injective.
\end{corollary}

\begin{proof}
Suppose that $s\in \scrF_\Zar(X)$ maps to $0$ in $\scrF_\Zar(K)$. By \Cref{lemma:injective}, the germs $s_x\in \scrF_x$ of $s$ vanish at all closed points of $X$, which yields $s=0$.
\end{proof}

\begin{corollary}\label{cor:inj}
For any nonempty open subscheme $i\colon V\hookrightarrow X$, the map $i^*\colon \scrF_\Zar(X)\to \scrF_\Zar(V)$ is injective.
\end{corollary}

\begin{proof}
We know that $K=k(V)$, hence \Cref{cor:gen-stalk} ensures that there are injections $\scrF_\Zar(X)\hookrightarrow \scrF(K)$ and $\scrF_\Zar(V)\hookrightarrow \scrF(K)$ induced by the generic point. Since $\scrF_\Zar(X)\hookrightarrow \scrF(K)$ factors through $\scrF_\Zar(V)$, the result follows.
\end{proof}

For the next lemma we will need to pass to the presheaf $\scrF^X$ on $\wt\Cor_K^X$, defined in \Cref{section:A1K}.

\begin{lemma}\label{lemma:exc-colim}
Let $x$ be a closed point in $\A^1_K$, and write $U_x\defeq\Spec(\calO_{\A^1_K,x})$ for its local scheme. Then the restriction map
\[
\dfrac{\scrF^X(\A^1_K\setminus x)}{\scrF^X(\A^1_K)}\xrightarrow{\cong}\dfrac{\scrF^X(U_x\setminus x)}{\scrF^X(U_x)}
\]
is an isomorphism.
\end{lemma}

\begin{proof}
We have
\[
\frac{\scrF^X(U_x\setminus x)}{\scrF^X(U_x)}=\varinjlim_{W\ni x}\frac{\scrF^X(W\setminus x)}{\scrF^X(W)},
\]
and so Zariski excision on $\A^1_K$ (\Cref{lemma:Zar-exc-A1K}) applied to the pair $x\in W\subseteq \A^1_K$ yields an isomorphism
\[
\dfrac{\scrF^X(\A^1_K\setminus x)}{\scrF^X(\A^1_K)}\xrightarrow{\cong}\dfrac{\scrF^X(W\setminus x)}{\scrF^X(W)}.
\]
The isomorphism is given by the pullback along the inclusion, so it is compatible with the transition maps in the directed system. It follows that the natural map from $\scrF^X(\A^1_K\setminus x)/\scrF^X(\A^1_K)$ to the colimit $\scrF^X(U_x\setminus x)/\scrF^X(U_x)$ is an isomorphism.
\end{proof}

\begin{lemma}\label{lemma:res-zar-site}
The sheafification map $\psi\colon \scrF^X(\A^1_{K})\to \scrF_\Zar^X(\A^1_{K})$ is an isomorphism.
\end{lemma}

\begin{proof}
Let $\xi$ be the generic point of $\A^1_K$. Since stalks remain the same after sheafification, the commutative diagram
\[\begin{tikzcd}
\scrF^X(K)\ar{r}{\cong}[swap]{p^*}\ar[dr]\ar[ddr,bend right,"\cong"] & \scrF^X(\A^1_K)\ar[d,"\psi"]\\
& \scrF^X_\Zar(\A^1_K)\ar{d}{i_0^*}\\
& \scrF^X_\Zar(K)
\end{tikzcd}\]
(in which $p^*$ is an isomorphism by \Cref{rem:Xhtpy2}) shows that $\psi$ is injective. It remains to show surjectivity.
 
Let $s\in \scrF^X_\Zar(\A^1_K)$ be a section, mapping to the germ $s_\xi\in \scrF^X_\xi$ at the generic point $\xi$ of $\A^1_K$ under the morphism $\xi^*\colon\scrF^X_\Zar(\A^1_K)\rightarrow \scrF^X_\xi$. As $\scrF^X_\xi=\varinjlim_{V\subseteq\A^1_K}\scrF^X(V)$, we can find a nonempty Zariski open $V\subseteq\A^1_K$ and a section $s'\in \scrF^X(V)$ such that $\psi(s')=s|_V\in\scrF^X_\Zar(V)$. Thus $s'_v=s_v$ for any $v\in V$. The idea from here is to extend the section $s'\in \scrF^X(V)$ to a global section $s''\in\scrF^X(\A^1_K)$.

We may assume that $V=\A^1_K\setminus x$, where $x$ is a closed point. Indeed, the general case follows by induction since $V$ is then the complement of finitely many closed points. For $U_x\defeq\Spec(\calO_{\A^1_K,x})$, the commutative diagram
\[\begin{tikzcd}
U_x\setminus x\ar[r,hook]\ar[d,hook] & \A^1_K\setminus x \ar[d,hook]\\
U_x\ar[r,hook] & \A^1_K
\end{tikzcd}\]
induces a commutative diagram
\[\begin{tikzcd}
\dfrac{\scrF^X(V)}{\scrF^X(\A^1_K)}\ar[r,"\cong"] & \dfrac{\scrF^X(U_x\setminus x)}{\scrF^X(U_x)}\\
\scrF^X(V)\ar[u]\ar[r] & \scrF^X(U_x\setminus x)=\scrF^X_\xi\ar[u]\\
\scrF^X(\A^1_K)\ar[u]\ar[r] & \scrF^X(U_x).\ar[u]
\end{tikzcd}\]
Here the upper horizontal arrow is an isomorphism by \Cref{lemma:exc-colim}. Moreover, note that $\scrF^X(U_x\setminus x)$ and $\scrF^X(U_x)$ are both stalks, as $\scrF^X(U_x\setminus x)=\scrF^X_\xi$. Thus we have the isomorphism
\[\dfrac{\scrF^X(U_x\setminus x)}{\scrF^X(U_x)}\cong\dfrac{\scrF^X_\Zar(U_x\setminus x)}{\scrF^X_\Zar(U_x)}.\]
We want to lift $s'\in \scrF^X(V)$ to $\scrF^X(\A^1_K)$, which is possible if and only if $s'$ maps to $0$ in the cokernel of the map $\scrF^X(\A^1_K)\to\scrF^X(V)$.
But $s'$ maps to $s_\xi$ under the map $\scrF^X(V)\to \scrF^X(U_x\setminus x)$ by the choice of $s'$. Moreover, $s_\xi\in \scrF^X_\xi$ is the image of the germ $s_x\in \scrF^X(U_x)$ of $s$ at $x$. Hence $s_\xi$ vanishes in $\scrF^X(U_x\setminus x)/\scrF^X(U_x)$. By the excision isomorphism we conclude that $s'$ vanishes in $\scrF^X(V)/\scrF^X(\A^1_K)$, and hence there is a section $s''\in \scrF^X(\A^1_K)$ such that $s''|_V=s'$. 

Finally, we need to check that $s''\in \scrF^X(\A^1_K)$ maps to $s$ under the morphism $\psi\colon\scrF^X(\A^1_K)\to \scrF^X_\Zar(\A^1_K)$. It suffices to show that the germs of $s''$ and $s$ coincide at every point of $\A^1_K$. For the points $v\in V$ we know that $s|_V=\psi(s')=\psi(s''|_V)$, so it remains to check that $s''_x=s_x$ in $\scrF^X(U_x)$. By \Cref{lemma:injective} we have an injection
\[
\xi^*\colon\scrF^X(U_x)\hookrightarrow\scrF^X_\xi.
\]
Since $\xi^*(s_x)=\xi^*(s_x'')=s_\xi$, we conclude that $s''_x=s_x$.
\end{proof}

\begin{theorem}\label{thm:Zar-htpy-inv}
If $\scrF\in\wt\PSh(k)$ is a homotopy invariant presheaf with $\rmMW$-transfers, then $\scrF_\Zar$ is homotopy invariant.
\end{theorem}

\begin{proof}
Let $i_0$ be the zero section $i_0\colon X\to X\times\A^1$, and write $K$ for the function field $k(X)$ of $X$. We then have $p\circ i_0=\id_X$, where $p\colon X\times \A^1\to X$ is the projection. Hence the induced map $i_0^*\colon \scrF_\Zar(X\times\A^1)\to \scrF_\Zar(X)$ is split surjective, and it remains to show that $i_0^*$ is injective. Consider the commutative diagram
\[\begin{tikzcd}
\scrF_\Zar(X\times\A^1)\ar[r, "(\eta\times1)^*"]\ar[d,swap,"i_0^*"] & \scrF_\Zar(\A^1_{K})\ar[d, "i_0^*"]\ar[equal]{r} & \scrF^X_\Zar(\A^1_{K})\ar{d}{(i_0^X)^*} \\
\scrF_\Zar(X)\ar[r,"\eta^*"] & \scrF_\Zar(K)\ar[equal]{r} & \scrF_\Zar^X(K),
\end{tikzcd}\]
where the right hand vertical map is the map on $\scrF_\Zar^X$ induced by the zero section. The homomorphisms $\eta^*$ and $(\eta\times1)^*$ are injective by \Cref{cor:gen-stalk} and \Cref{cor:inj}, respectively. Now, notice that $(i_0^X)^*=i_0^*$ by the definition of the presheaf $\scrF^X$. Hence the right hand square is commutative, and thus it suffices to show that the map $i_0^*\colon\scrF_\Zar^X(\A^1_K)\to\scrF^X_\Zar(K)$ is injective. Using \Cref{lemma:res-zar-site} along with homotopy invariance of the presheaf $\scrF^X$ (see \Cref{rem:Xhtpy2}) we find
\[
\scrF^X_\Zar(\A^1_{K})\cong\scrF^X(\A^1_{K})\cong\scrF^X(K)=\scrF^X_\Zar(K).
\]
Hence the right hand vertical map is an isomorphism. We conclude that $i_0^*\colon \scrF_\Zar(X\times\A^1)\to \scrF_\Zar(X)$ is injective.
\end{proof}

\subsection*{Homotopy invariance of $\scrF_\Nis$}
We proceed to prove homotopy invariance of the associated Nisnevich sheaf $\scrF_\Nis$, the proof being similar to the one for Zariski sheafification using Nisnevich excision. If $A$ is a local ring, let $A^h$ denote the henselization of $A$. We may write $\scrF(\Spec(\calO_{X,x}^h))$ or $\scrF(\calO_{X,x}^h)$ for the stalk of $\scrF$ at $x$ in the Nisnevich topology. Thus $\scrF(\calO_{X,x}^h)=\varinjlim_V\scrF(V)$, where the colimit runs over the filtered system of étale neighborhoods of $x$ in $X$, i.e., étale morphisms $p\colon V\to X$ such that $p^{-1}(x)\cong x$.

\begin{lemma}\label{lemma:Nis-injective}
For $U_x^h\defeq\Spec(\calO_{X,x}^h)$, the natural map $\scrF(U_x^h)\to \scrF(k(U_x^h))$ is injective.
\end{lemma}

\begin{proof}
Suppose that $s\in \scrF(U_x^h)$ maps to $0$ in $\scrF(k(U_x^h))$. This means that there is some étale neighborhood $p\colon W\to X$ such that $s|_W=0$. Replacing $W$ by its open image, we may assume that $W\subseteq X$. Let $Z$ be the closed complement of $W$ in $X$. If $x\in W$ then $s=0$ in $\scrF(U_x^h)$; if not then $x\in Z$, and thus \Cref{cor:vanishing-section} shows that $s|_V=0$ for some Zariski neighborhood $V$ of $x$. Since $V$ is also an étale neighboorhood, it follows that $s=0$ in $\scrF(U_x^h)$.
\end{proof}

The next two corollaries follow from \Cref{lemma:Nis-injective} similarly to the Zariski case.

\begin{corollary}\label{cor:Nis-gen-stalk}
The map $\eta^*\colon \scrF_\Nis(X)\to \scrF_\Nis(K)$ is injective.
\end{corollary}

\begin{corollary}\label{cor:Nis-inj}
For any nonempty open subscheme $i\colon V\hookrightarrow X$, the map $i^*\colon \scrF_\Nis(X)\to \scrF_\Nis(V)$ is injective.
\end{corollary}

\begin{lemma}\label{lemma:Nis-exc-colim}
Let $x$ be a closed point in $\A^1_K$. Write $U_x\defeq\Spec(\calO_{\A^1_K,x})$ and $U_x^h\defeq\Spec(\calO^h_{\A^1_K,x})$. Then there is a natural isomorphism
\[
\dfrac{\scrF^X(U_x\setminus x)}{\scrF^X(U_x)}\xrightarrow{\cong}\dfrac{\scrF^X(U_x^h\setminus x)}{\scrF^X(U_x^h)}.
\]
\end{lemma}

\begin{proof}
We have
\begin{align*}
\dfrac{\scrF^X(U_x^h\setminus x)}{\scrF^X(U_x^h)} &=\varinjlim_{W\to\A^1_K}\dfrac{\scrF^X(W\setminus x)}{\scrF^X(W)}\\
&=\varinjlim_{W\to\A^1_K}\varinjlim_{W'\subseteq W}\dfrac{\scrF^X(W'\setminus x)}{\scrF^X(W')}\\
&=\varinjlim_{W\to\A^1_K}\dfrac{\scrF^X(\Spec(\calO_{W,x})\setminus x)}{\scrF^X(\Spec(\calO_{W,x}))}\\
&\xrightarrow{\cong} \varinjlim_{W\to\A^1_K}\dfrac{\scrF^X(U_x\setminus x)}{\scrF^X(U_x)}=\dfrac{\scrF^X(U_x\setminus x)}{\scrF^X(U_x)}.
\end{align*}
Here $W$ runs over all étale neighborhoods of $x$ in $\A^1_K$; $W'$ runs over all Zariski open neighborhoods of $x$ in $W$; and the fourth isomorphism is given by Nisnevich excision.
\end{proof}

\begin{lemma}\label{lemma:sheafification}
The sheafification map $\psi\colon \scrF^X(\A^1_{K})\to \scrF^X_\Nis(\A^1_{K})$ is an isomorphism.
\end{lemma}

\begin{proof}
Let $\xi$ be the generic point of $\A^1_K$.
By the same reasoning as in the proof of \Cref{lemma:res-zar-site}, the map $\scrF^X(\A^1_K)\to \scrF^X_\Nis(\A^1_K)$ is injective,
and it remains to show surjectivity. Let $s\in \scrF^X_\Nis(\A^1_K)$ be a section. Since the stalks $\scrF^X_\xi$ and $(\scrF^X_\Nis)_\xi$ coincide, there exists a Zariski open subscheme $V\subseteq\A^1_K$ and a section $s'\in \scrF^X(V)$ such that $\psi(s')=s|_V$ in $\scrF^X_\Nis(V)$. We wish to extend $s'\in \scrF^X(V)$ to a global section $s''\in \scrF^X(\A^1_K)$. Considering one point at a time, we may assume that $V=\A^1_K\setminus x$ for some closed point $x$. Let $U_x\defeq\Spec(\calO_{\A^1_K,x})$ and $U_x^h\defeq\Spec(\calO_{\A^1_K,x}^h)$. A lift of $s'$ to a section of $\scrF^X(\A^1_K)$ exists if and only if $s'$ maps to $0$ in the quotient $\scrF^X(V)/\scrF^X(\A^1_K)$. Consider the sequence
\[
\dfrac{\scrF^X(V)}{\scrF^X(\A^1_K)} \xrightarrow{\cong} \dfrac{\scrF^X(U_x\setminus x)}{\scrF^X(U_x)} \xrightarrow{\cong} \dfrac{\scrF^X(U_x^h\setminus x)}{\scrF^X(U_x^h)} \xrightarrow{\cong} \dfrac{\scrF^X_\Nis(U_x^h\setminus x)}{\scrF^X_\Nis(U_x^h)}.
\]
Here the left hand map is an isomorphism by \Cref{lemma:exc-colim}; the middle map is an isomorphism by \Cref{lemma:Nis-exc-colim}; and the right hand map is an isomorphism since both $\scrF^X(U_x^h\setminus x)$ and $\scrF^X(U_x^h)$ are stalks in the Nisnevich topology. Thus it is enough to show that $s'\in \scrF^X(V)$ maps to $0$ in $\scrF^X_\Nis(U_x^h\setminus x)/\scrF^X(U_x^h)$. But this follows from the commutativity of the diagram
\[\begin{tikzcd}
\scrF^X(V)\ar[r] & \scrF^X(U_x^h\setminus x)\\
\scrF^X(\A^1_K)\ar[r]\ar[u] & \scrF^X(U_x^h).\ar[u,hook]
\end{tikzcd}\]
Hence we can lift $s'$ to $s''\in \scrF^X(\A^1_K)$. It remains to check that $s''$ maps to $s\in \scrF^X_\Nis(\A^1_K)$. Knowing that $\psi(s''|_V)=s|_V\in \scrF^X_\Nis(V)$, it remains to show that $s''_x=s_x\in \scrF^X_\Nis(U_x^h)=\scrF^X(U_x^h)$. As $\scrF^X(U_x^h)$ injects into $\scrF^X(U_x^h\setminus x)=\scrF^X(k(U_x^h))$ by \Cref{lemma:Nis-injective}, it is sufficient to prove the equality in the latter stalk. This follows from the commutativity of the above diagram, using that both $s$ and $s''$ map to $s|_V$ in $\scrF^X_\Nis(V)$.\end{proof}

\begin{theorem}\label{thm:main1}
If $\scrF$ is a homotopy invariant presheaf on $\wt\Cor_k$, then $\scrF_\Nis$ is also homotopy invariant.
\end{theorem}

\begin{proof}We must show that the map $i_0^*\colon \scrF_\Nis(X\times\A^1)\to\scrF_\Nis(X)$ induced by the zero section is injective. As in the proof of \Cref{thm:Zar-htpy-inv} we consider the commutative diagram
\[\begin{tikzcd}
\scrF_\Nis(X\times\A^1)\ar[d,swap,"i_0^*"]\ar[r,"(\eta\times1)^*"]  & \scrF_\Nis^X(\A^1_{K})\ar{d}{i_0^*} \\
\scrF_\Nis(X)\ar[r,"\eta^*"] & \scrF_\Nis^X(K).
\end{tikzcd}\]
The homomorphisms $\eta^*$ and $(\eta\times1)^*$ are injective by \Cref{cor:Nis-gen-stalk} and \Cref{cor:Nis-inj}, respectively. Using \Cref{lemma:sheafification} along with homotopy invariance of the presheaf $\scrF^X$, we find that the right hand vertical map $i_0^*$ is an isomorphism. Hence $i_0^*\colon\scrF_\Nis(X\times\A^1)\to\scrF_\Nis(X)$ is injective.
\end{proof}

\providecommand{\bysame}{\leavevmode\hbox to3em{\hrulefill}\thinspace}
\providecommand{\MR}{\relax\ifhmode\unskip\space\fi MR }
\providecommand{\MRhref}[2]{%
  \href{http://www.ams.org/mathscinet-getitem?mr=#1}{#2}
}
\providecommand{\href}[2]{#2}

\end{document}